\documentclass[a4paper,11pt]{article}

\usepackage{graphicx}
\usepackage{amssymb,amsmath,amsthm,times,mathrsfs,fullpage}
\usepackage{tikz,wrapfig,multirow}
\usetikzlibrary{patterns,arrows,decorations.pathreplacing}
\usepackage[pdftex]{hyperref}
\usepackage{epstopdf}
\hypersetup{plainpages=True, pdfstartview=FitH, bookmarksopen=true, colorlinks=true,linkcolor=blue,citecolor=blue}

\allowdisplaybreaks

\def\qed{{\quad\rule{1mm}{3mm}\,}}

\usepackage{float}
\usetikzlibrary{calc}
\usetikzlibrary{decorations.markings}
\usetikzlibrary {backgrounds}
\usetikzlibrary{shapes,arrows,arrows.meta}
\usepackage{subcaption}
\begin{document}

\pgfdeclarelayer{background}
\pgfdeclarelayer{foreground}
\pgfsetlayers{background,main,foreground}

\newtheorem{thm}{Theorem}
\newtheorem*{thm*}{Theorem}
\newtheorem{cor}[thm]{Corollary}
\newtheorem{lmm}[thm]{Lemma}
\newtheorem{conj}[thm]{Conjecture}
\newtheorem{pro}[thm]{Proposition}
\newtheorem*{pro*}{Proposition}
\theoremstyle{definition}\newtheorem{Def}{Definition}
\theoremstyle{remark}\newtheorem{Rem}{Remark}

\title{Counting Phylogenetic Networks with Few Reticulation Vertices: Galled and Reticulation-Visible Networks}
\author{Yu-Sheng Chang and Michael Fuchs\\
    Department of Mathematical Sciences\\
    National Chengchi University\\
    Taipei 116\\
    Taiwan}
\maketitle

\begin{abstract}
We give exact and asymptotic counting results for the number of galled networks and reticulation-visible networks with few reticulation vertices. Our results are obtained with the component graph method, which was introduced by L. Zhang and his coauthors, and generating function techniques. For galled networks, we in addition use analytic combinatorics. Moreover, in an appendix, we consider maximally reticulated reticulation-visible networks and derive their number, too.
\end{abstract}

\section{Introduction and Results}\label{intro}

This is the fourth of a series of papers which is concerned with the enumeration of phylogenetic networks with few reticulation vertices from a fixed class of phylogenetic networks. In the first three papers, we considered {\it normal} and {\it tree-child} networks; see \cite{FuGiMa1,FuGiMa2,FuHuYu}. More precisely, in \cite{FuGiMa1,FuHuYu}, we proposed two methods for deriving asymptotic counting results when the number of reticulation vertices is fixed and the number of leaves tends to infinity. In \cite{FuGiMa2}, we corrected mistakes from the approach from \cite{FuGiMa1} and derived exact counting formulas. For instance, one of the results of \cite{FuGiMa2} reads as follows.

\begin{thm*}[\cite{FuGiMa2}]
For the number $N_{\ell,2}$ of normal networks with $\ell$ leaves and two reticulation vertices,
\[
N_{\ell,2}=\frac{(3\ell-4)(\ell^2+11\ell+6)}{3}(2\ell-1)!!-2^{\ell}(\ell+2)(3\ell-4)\ell!.
\]
\end{thm*}

This result solved an open problem from \cite{CaZh} where exact counting results for tree-child networks were obtained. However, the above result was derived via generating function techniques, whereas the results from \cite{CaZh} were obtained with the {\it component graph method}. This method as well as generating function techniques will also play an important role in the current paper.

The main purpose of this paper is to prove exact and asymptotic counting results for the classes of galled networks and reticulation-visible networks. We will start with definitions.

\begin{Def}[Phylogenetic Networks]
A (rooted) {\it phylogenetic network} of size $\ell$ is a rooted simple directed acyclic graph (rooted DAG) whose vertices belong to the following three categories:
\begin{itemize}
\item[(i)] A (unique) {\it root} which has indegree $0$ and outdegree $1$;
\item[(ii)] {\it Leaves} which have indegree $1$ and outdegree $0$ and which are bijectively labeled with elements from the set $\{1,\ldots,\ell\}$;
\item[(iii)] {\it Internal vertices} which have indegree and outdegree at least $1$ but not both equal to $1$.
\end{itemize}
\end{Def}

Internal vertices with indegree at least $2$ are called {\it reticulation vertices}; all other internal vertices are called {\it tree vertices}. Moreover, a phylogenetic network is called {\it binary} if all internal vertices have total degree equal to $3$. If not stated otherwise, networks will subsequently always be binary.

Many subclasses of the class of phylogenetic networks have been proposed; see the recent survey \cite{KoPoKuWi} for most of them. The ones of relevance for this paper are defined next.

To be able to state the definitions, we need two notions: first, a {\it tree cycle} is a pair of edge disjoint paths from a common tree vertex to a common reticulation vertex with all remaining vertices being tree vertices; second, a vertex in a phylogenetic network is called {\it visible}, if there exists a leaf such that any path from the root to the leaf must contain the vertex.

\begin{Def}\label{def-net-classes}
A phylogenetic network is called:
\begin{itemize}
\item[(i)] {\it Tree-child network} if every non-leaf vertex has at least one child which is not a reticulation vertex;
\item[(ii)] {\it Normal} if it is tree-child and has no {\it shortcuts}, i.e., the two parents of a reticulation vertex are not in an ancestor-descendant relationship.
\item[(iii)] {\it Galled} if each reticulation vertex is in a (necessarily unique) tree cycle;
\item[(iv)] {\it Reticulation-visible} if each reticulation vertex is visible.
\end{itemize}
\end{Def}

The set-inclusion relationship of these classes is depicted in Figure~\ref{network-classes}, where the class at the bottom is the class of {\it phylogenetic trees} which are networks without reticulation vertices.

\begin{figure}[!t]
\begin{center}
\begin{tikzpicture}[every node/.style={rectangle,rounded
corners=0.15cm,thick,draw=black!50,top color=white, bottom
color=black!20,text centered,scale=0.75},line width=0.15mm,scale=0.7]
\draw (0cm,0cm) node[text width=2.2cm] (1) {Phylogenetic networks};
\draw (0cm,-2.2cm) node[text width=2.5cm] (2) {Reticulation-visible networks};
\draw (2.4cm,-4.5cm) node[text width=1.5cm] (3) {Galled networks};
\draw (-2.5cm,-4.5cm) node[text width=2cm] (4) {Tree-child networks};
\draw (-2.5cm,-6.7cm) node[text width=1.5cm] (5) {Normal networks};
\draw (0cm,-8.8cm) node[text width=2.2cm] (6) {Phylogenetic trees};

\draw (1)--(2); \draw (2)--(3); \draw (3)--(6);
\draw (2)--(4); \draw (4)--(5); \draw (5)--(6);
\end{tikzpicture}
\end{center}
\caption{Hasse diagram of the network classes considered in this paper.}\label{network-classes}
\end{figure}
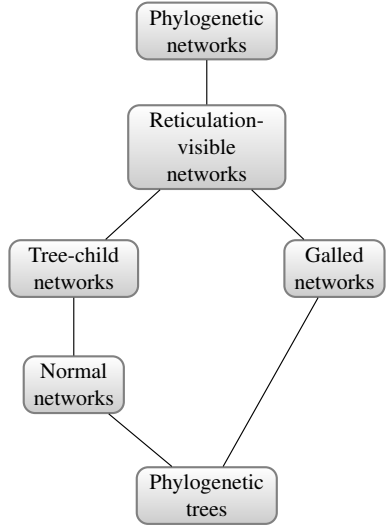

We are interested in enumeration results for these classes. Throughout the paper, we will use the following notation for the number of networks from a class with $\ell$ leaves and $k$ reticulation vertices. (The notation for the number of normal networks already appeared in the above theorem.)

\begin{center}
\begin{tabular}{c|c}
network class & number of networks \\
\hline
normal networks & $N_{\ell,k}$ \\
tree-child networks & ${\rm TC}_{\ell,k}$ \\
galled networks & ${\rm GN}_{\ell,k}$ \\
reticulation-visible networks & ${\rm RV}_{\ell,k}$ \\
phylogenetic networks & ${\rm PN}_{\ell,k}$
\end{tabular}
\end{center}

The following asymptotic result for these numbers is known. For fixed $k$,
\[
{\rm PN}_{\ell,k}\sim{\rm TC}_{\ell,k}\sim{\rm N}_{\ell,k}\sim\frac{2^{k-1}\sqrt{2}}{k!}\left(\frac{2}{e}\right)^{\ell}\ell^{\ell+2k-1},\qquad (\ell\rightarrow\infty).
\]
The first asymptotic equivalence was proved in \cite{Ma} and the second and third in \cite{FuGiMa1,FuGiMa2} (without the closed-form expression for the leading constant) and in \cite{FuHuYu} (with the closed-form expression for the leading constant). An easy consequence of these results is the following.

\begin{cor}
For fixed $k$,
\[
{\rm RV}_{\ell,k}\sim\frac{2^{k-1}\sqrt{2}}{k!}\left(\frac{2}{e}\right)^{\ell}\ell^{\ell+2k-1},\qquad (\ell\rightarrow\infty).
\]
\end{cor}

Thus, for fixed $k$, the asymptotics of the number of networks for all classes from Definition~\ref{def-net-classes} is known {\it except} for the class of galled networks. Here, somehow surprisingly, again the same result holds.

\begin{thm}\label{asymp-GN}
For fixed $k$,
\[
{\rm GN}_{\ell,k}\sim\frac{2^{k-1}\sqrt{2}}{k!}\left(\frac{2}{e}\right)^{\ell}\ell^{\ell+2k-1},\qquad (\ell\rightarrow\infty).
\]
\end{thm}

We will prove this result below. Our method will also allow us to obtain closed-form expressions for ${\rm GN}_{\ell,k}$ for small $k$. We state the results for $k=2$ and $k=3$.

\begin{thm}\label{exact-GN}
We have,
\[
{\rm GN}_{\ell,2}=\frac{6\ell^4+31\ell^3+30\ell^2-7\ell-9}{3}(2\ell-3)!!-2^{\ell-2}(7\ell+10)(\ell+1)!
\]
and
\begin{align*}
{\rm GN}_{\ell,3}= &\frac{140\ell^6+3184\ell^5+17195\ell^4+34125\ell^3+19475\ell^2-8599\ell-6090}{105}(2\ell-3)!! \\
&\qquad\qquad-2^{\ell-5}\frac{225\ell^3+2045\ell^2+5878\ell+5448}{3}(\ell+1)!.
\end{align*}
\end{thm}

\begin{Rem}
The formula for $k=2$ already appeared in \cite{CaZh} (with typos which we have corrected here).
\end{Rem}

Moreover, we will also give similar formulas for ${\rm RV}_{\ell,k}$ for small $k$.

\begin{thm}\label{exact-RV}
We have,
\[
{\rm RV}_{\ell,2}=\frac{6\ell^4+7\ell^3+6\ell^2-\ell-3}{3}(2\ell-3)!!-2^{\ell-1}(2\ell^2+2\ell+1)\ell!
\]
and
\begin{align*}
{\rm RV}_{\ell,3}=&\frac{4\ell^6+20\ell^5+33\ell^4-32\ell^3-76\ell^2+12\ell+12}{3}(2\ell-3)!!\\
&\qquad-2^{\ell-4}\frac{48\ell^4+175\ell^3+99\ell^2-262\ell-264}{3}\ell!.
\end{align*}
\end{thm}

\begin{Rem}
Note that for $k=0$, all the numbers coincide, i.e.,
\[
{\rm PN}_{\ell,0}={\rm RV}_{\ell,0}={\rm GN}_{\ell,0}={\rm TC}_{\ell,0}={\rm N}_{\ell,0}=(2\ell-3)!!,
\]
where the last number is the (well-known) number of phylogenetic trees with $\ell$ leaves. Also, for $k=1$, all the numbers, except the number of normal network, coincide:
\[
{\rm PN}_{\ell,1}={\rm RV}_{\ell,1}={\rm GN}_{\ell,1}={\rm TC}_{\ell,1}=\ell(2\ell-1)!!-2^{\ell-1}\ell!,
\]
where the last result was derived, e.g., in \cite{CaZh}. ($N_{\ell,1}$ is known as well; see, e.g., \cite{Zh}.)
\end{Rem}

Finally, we will consider maximal reticulated networks. Here, the following sharp upper bounds for the number of reticulation vertices $k$ have been proved.

\begin{center}
\begin{tabular}{c|c}
network class & maximal number of reticulation vertices \\
\hline
normal networks & $\ell-2$ \\
tree-child networks & $\ell-1$ \\
galled networks & $2\ell-2$ \\
reticulation-visible networks & $3\ell-3$
\end{tabular}
\end{center}

The result for normal networks first appeared in \cite{Wi}; the result for tree-child networks was discovered by many authors (and is, in fact, easy to prove). For galled networks, the optimal upper bound was proved, e.g., in \cite{GuGuZh,GuZh}; see also \cite{FuYuZh2}. Finally, for reticulation-visible networks, in \cite{GaGuLaViZh}, the upper bound $4\ell-4$ was established and then this upper bound was reduced to the optimal one independently in \cite{BoSe} and \cite{GuZh}. Here, we will give a simplified proof which in addition also gives the number of maximal reticulated reticulation-visible networks.

\begin{thm}\label{max-ret-RV}
The maximum number of reticulation vertices in a reticulation-visible networks with $\ell$ leaves equals $3\ell-3$. Moreover,
\[
{\rm RV}_{\ell,3\ell-3}={\rm TC}_{\ell,\ell-1}=\Theta\left(\ell^{-2/3}e^{a_1(3\ell)^{1/3}}\left(\frac{12}{e^2}\right)^{\ell}\ell^{2\ell}\right),
\]
where $a_1$ is the largest root of the Airy function of the first kind.
\end{thm}
\begin{Rem}
The asymptotic result follows from the main result in \cite{FuYuZh1}.
\end{Rem}

We conclude the introduction by an outline of the paper. The main method of proof will be the component graph method which was introduced by Zhang and his coauthors. We will recall this method and some related results in the next section. In Section~\ref{gall-net}, we will then combine it with generating function techniques and tools from analytic combinatorics to prove Theorem~\ref{asymp-GN} and Theorem~\ref{exact-GN} for galled networks. Section~\ref{ret-vis-net} will contain the proof of Theorem~\ref{exact-RV} for reticulation-visible networks. Theorem~\ref{max-ret-RV} is not really concerned with phylogenetic networks with few reticulation vertices and therefore does not really fall into the scope of this paper. Therefore, we will prove it in the appendix. The paper will be concluded with some final remarks in Section~\ref{con}.

\section{The Component Graph Method}\label{cgm}

The component graph method was introduced by Zhang and used by him and his coauthors to solve algorithmic problems for phylogenetic networks; see \cite{CaZh,GuGuZh,GuRaZh,GuYaZh}. For instance, it was used to algorithmically solve the counting problem for tree-child networks in \cite{CaZh} and galled networks in \cite{GuRaZh}.

The central object of the method are {\it component graphs}. First, for a given phylogenetic network, its {\it tree components} are obtained by removing the two incoming edges of every reticulation vertex. (These edges are called {\it reticulation edges}.) Using these objects, component graphs are defined as follows.

\begin{Def}[Component Graph]
The {\it component graph} of a phylogenetic network $N$ has a vertex for each tree-component. These vertices are then connected via edges according to how the tree-components are connected in $N$ via the reticulation edges. Finally, we attach the leaves (with their labels) in the tree-components to their vertices in the component graph except when the tree-component has only one leaf and corresponds to a terminal vertex in the component graph with incoming edges a double edge, in which case we use the label of the leaf to label the terminal vertex.
\end{Def}

See Figures~\ref{fig:example_RV} for a phylogenetic network and its component graph. Note that we have replaced double edges by single edges but have indicated that they are double edges by placing arrows on them.

\newpage

%% \label{fig:example_RV}
%% \input{tikz/example_RV_componentGraph.txt}
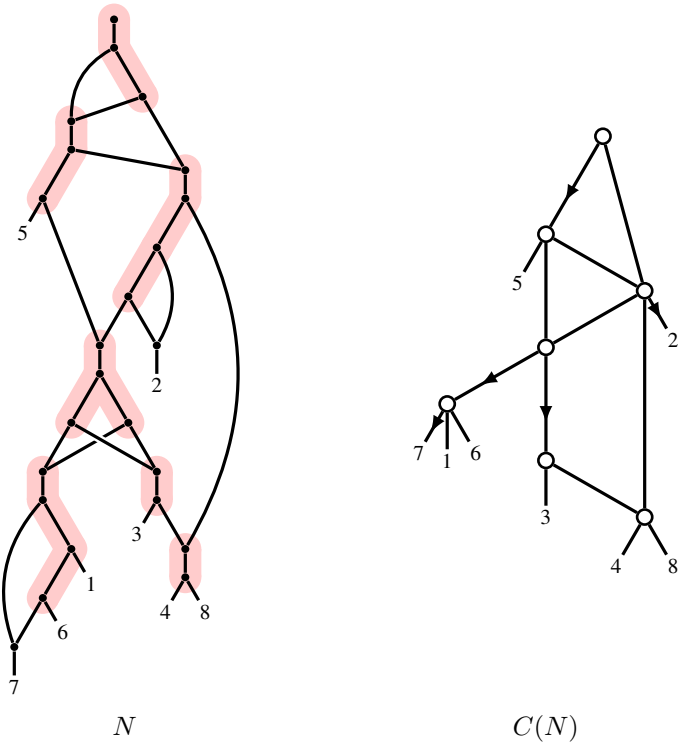
\begin{figure}[H]
\centering
\tikzset{
    styNode/.style={
        circle, draw, line width=1pt, inner sep = 2pt, font=\scriptsize
    },
    smallNode/.style={
        circle, fill=black, inner sep=0pt, minimum size=3pt
    }
}
\newcommand{\doubleEdge}[2]{
    %%\draw[-{Latex[length=3mm]}, line width=3pt] (#1) -- (#2);
    %%\draw[white, line width=1pt, shorten >= 3mm] (#1) -- (#2);
    \draw[line width=1.2pt, decoration={markings,mark=at position 0.7 with {\arrow{latex[length=3mm]}}}, postaction={decorate}] (#1) -- (#2);
}
\newcommand{\singleEdge}[2]{
    %% \draw[-{Latex[length=2mm]}, line width=1.5pt] (#1) -- (#2);
    \draw[line width=1.2pt] (#1) -- (#2);
}
\newcommand{\bendEdge}[3]{
    %% \draw[-{Latex[length=2mm]}, line width=1.5pt] (#2) to [bend right=#1] (#3);
    \draw[line width=1.2pt] (#2) to [bend right=#1] (#3);
}
\newcommand{\parH}{1.5}
\newcommand{\parh}{0.5*\parH}

\newcommand{\myTheta}{30}
\newcommand{\addLeaf}[5]{
    %% add a leaf, (#2), labelled by #3 from the node (#1) at angle #4
    %% #1 .. label of A node
    %%\node[smallNode] (#2) at ($(#1) + (-90+#4:0.5*\parh)$) {};
    \node[line width=1pt,inner sep = 1pt, font=\scriptsize] (#2) at ($(#1) + (-90+#4:#5)$) {#3};
    \singleEdge{#1}{#2}
}
\newcommand{\addRetiLeaf}[5]{
    %% add a leaf, (#2), labelled by #3 from the node (#1) at angle #4
    %% #1 .. label of A node
    %%\node[smallNode] (#2) at ($(#1) + (-90+#4:0.5*\parh)$) {};
    \node[line width=1pt,inner sep = 1pt, font=\scriptsize] (#2) at ($(#1) + (-90+#4:#5)$) {#3};
    \doubleEdge{#1}{#2}
}
\begin{subfigure}[b]{0.33\textwidth}
    \centering
    \begin{tikzpicture}
        %% nodes
        %% A
        \node[smallNode] (a1) at (0,0) {};
        \node[smallNode] (a2) at ($(a1) + (-90:0.5*\parh)$) {};
        \node[smallNode] (a3) at ($(a2) + (-90+\myTheta:1.0*\parh)$) {};
        %% B
        \node[smallNode] (b1) at ($(a2) + (-90-\myTheta:1.5*\parh)$) {};
        \node[smallNode] (b2) at ($(b1) + (-90:0.5*\parh)$) {};
        \node[smallNode] (b3) at ($(b2) + (-90-\myTheta:1.0*\parh)$) {};
        %% C
        \node[smallNode] (c1) at ($(a3) + (-90+\myTheta:1.5*\parh)$) {};
        \node[smallNode] (c2) at ($(c1) + (-90:0.5*\parh)$) {};
        \node[smallNode] (c3) at ($(c2) + (-90-\myTheta:1.0*\parh)$) {};
        \node[smallNode] (c4) at ($(c3) + (-90-\myTheta:1.0*\parh)$) {};
        %% D
        \node[smallNode] (d1) at ($(c4) + (-90-\myTheta:1.0*\parh)$) {};
        \node[smallNode] (d2) at ($(d1) + (-90:0.5*\parh)$) {};
        \node[smallNode] (d3) at ($(d2) + (-90-\myTheta:1.0*\parh)$) {};
        \node[smallNode] (d4) at ($(d2) + (-90+\myTheta:1.0*\parh)$) {};
        %% E
        \node[smallNode] (e1) at ($(c4) + (-90+\myTheta:1.0*\parh)$) {};
        %% F
        \node[smallNode] (f1) at ($(d3) + (-90-\myTheta:1.0*\parh)$) {};
        \node[smallNode] (f2) at ($(f1) + (-90:0.5*\parh)$) {};
        \node[smallNode] (f3) at ($(f2) + (-90+\myTheta:1.0*\parh)$) {};
        \node[smallNode] (f4) at ($(f3) + (-90-\myTheta:1.0*\parh)$) {};
        %% G
        \node[smallNode] (g1) at ($(d4) + (-90+\myTheta:1.0*\parh)$) {};
        \node[smallNode] (g2) at ($(g1) + (-90:0.5*\parh)$) {};
        %% H
        \node[smallNode] (h1) at ($(f4) + (-90-\myTheta:1.0*\parh)$) {};
        %% I
        \node[smallNode] (i1) at ($(g2) + (-90+\myTheta:1.0*\parh)$) {};
        \node[smallNode] (i2) at ($(i1) + (-90:0.5*\parh)$) {};

        %% single edges
        \foreach \x/\y in {
            a1/a2,a2/a3,
            a3/b1,b1/b2,b2/b3,
            a3/c1,b2/c1,c1/c2,c2/c3,c3/c4,
            b3/d1,c4/d1,d1/d2,d2/d3,d2/d4,
            c4/e1,
            d3/f1,d4/f1,f1/f2,f2/f3,f3/f4,
            d4/g1,g1/g2,
            f4/h1,
            g2/i1,i1/i2}{
            \singleEdge{\x}{\y}
        }
        \draw[white, line width=3pt, shorten >= 5mm, shorten <= 2mm] (d3) -- (g1);
        \singleEdge{d3}{g1}

        %% bend edges
        \foreach \t/\x/\y in {
            1/a2/b1,
            -1/c3/e1,
            1/f2/h1,
            -1/c2/i1}{
            \bendEdge{\t*\myTheta}{\x}{\y}
        }

        %% leaves
        \foreach \x/\n/\t in {
            f3/1/1,
            e1/2/0,
            g2/3/-1,
            i2/4/-1,
            b3/5/-1,
            f4/6/1,
            h1/7/0,
            i2/8/1}{
            \addLeaf{\x}{L\n}{\n}{\t*\myTheta}{0.7*\parh}
        }

        %% background
        \begin{scope}[on background layer={color=yellow}]
            \foreach \x in {
                a1, a2, a3,
                b1, b2, b3,
                c1, c2, c3, c4,
                d1, d2, d3, d4,
                f1, f2, f3, f4,
                g1, g2,
                i1, i2}{
                \node[circle, fill=white!80!red, minimum size=12pt] at (\x) {};
            }
            \foreach \x/\y in {
                a1/a2,a2/a3,
                b1/b2,b2/b3,
                c1/c2,c2/c3,c3/c4,
                d1/d2,d2/d3,d2/d4,
                f1/f2,f2/f3,f3/f4,
                g1/g2,
                i1/i2}{
                \draw[draw=white!80!red, line width=12pt] (\x.center) -- (\y.center);
            }
        \end{scope}
    \end{tikzpicture}
    \caption*{$N$}
\end{subfigure}
\renewcommand{\myTheta}{30}
\newcommand{\deltaH}{1.6}
\begin{subfigure}[b]{0.33\textwidth}
    \centering
    \begin{tikzpicture}
        %% nodes
        \node[styNode] (A) at (  0,  0) {};
        \node[styNode] (B) at ($(A) + (-90-\myTheta:\parH)$) {};
        \node[styNode] (C) at ($(B) + (-90+2*\myTheta:\parH)$) {};
        \node[styNode] (D) at ($(B) + (-90-0*\myTheta:\parH)$) {};
        %%\node[styNode] (E) at ($(C) + (-90+\myTheta:\parH)$) {};
        \node[styNode] (F) at ($(D) + (-90-2*\myTheta:\parH)$) {};
        \node[styNode] (G) at ($(D) + (-90+0*\myTheta:\parH)$) {};
        %%\node[styNode] (H) at ($(F) + (-90-\myTheta:\parH)$) {};
        \node[styNode] (I) at ($(G) + (-90+2*\myTheta:\parH)$) {};

        %% double edges
        \foreach \x/\y in {
            A/B,
            %%C/E,
            D/F,
            D/G}{ %% F/H
            \doubleEdge{\x}{\y}
        }

        %% single edges
        \foreach \x/\y in {
            A/C, B/C,
            B/D, C/D,
            G/I, C/I}{ %%
            \singleEdge{\x}{\y}
        }

        %% leaves
        \foreach \x/\n/\t in {
            F/1/0,
            %%E/2/0,
            G/3/0,
            I/4/-1,
            B/5/-1,
            F/6/1,
            I/8/1}{
            %% H/7/0
            \addLeaf{\x}{L\n}{\n}{\t*\myTheta}{\parh}
        }
        %% Rericulations leaves
        \foreach \x/\n/\t in {
            F/7/-1,
            C/2/1}{
            \addRetiLeaf{\x}{L\n}{\n}{\t*\myTheta}{\parh}
        }

        %% to controll vertically aligned
        \node[] (TOP) at ($(A) + ( 90:\deltaH)$) {};
        \node[] (BTM) at ($(L4) + (-90:\deltaH)$) {};

    \end{tikzpicture}
    \caption*{$C(N)$}
\end{subfigure}
\vspace*{0.15cm}\caption{Left: A reticulation-visible network with $8$ leaves and $8$ reticulation vertices (with the tree-components highlighted). Right: The component graph of the network. Note that it is tree-child network.}
\label{fig:example_RV}
\end{figure}
\vspace*{0.35cm}
%% \label{fig:example_GN}
%% \input{tikz/example_GN_componentGraph.txt}
\begin{figure}[H]
\centering
\tikzset{
    styNode/.style={
        circle, draw, line width=1pt, inner sep = 2pt, font=\scriptsize
    },
    smallNode/.style={
        circle, fill=black, inner sep=0pt, minimum size=3pt
    }
}
\newcommand{\doubleEdge}[2]{
    %%\draw[-{Latex[length=3mm]}, line width=3pt] (#1) -- (#2);
    %%\draw[white, line width=1pt, shorten >= 3mm] (#1) -- (#2);
    \draw[line width=1.2pt, decoration={markings,mark=at position 0.7 with {\arrow{latex[length=3mm]}}}, postaction={decorate}] (#1) -- (#2);
}
\newcommand{\singleEdge}[2]{
    %% \draw[-{Latex[length=2mm]}, line width=1.5pt] (#1) -- (#2);
    \draw[line width=1.2pt] (#1) -- (#2);
}
\newcommand{\bendEdge}[3]{
    %% \draw[-{Latex[length=2mm]}, line width=1.5pt] (#2) to [bend right=#1] (#3);
    \draw[line width=1.2pt] (#2) to [bend right=#1] (#3);
}
\newcommand{\parH}{1.5}
\newcommand{\parh}{0.5*\parH}

\newcommand{\myTheta}{30}
\newcommand{\addLeaf}[5]{
    %% add a leaf, (#2), labelled by #3 from the node (#1) at angle #4
    %% #1 .. label of A node
    %%\node[smallNode] (#2) at ($(#1) + (-90+#4:0.5*\parh)$) {};
    \node[line width=1pt,inner sep = 1pt, font=\scriptsize] (#2) at ($(#1) + (-90+#4:#5)$) {#3};
    \singleEdge{#1}{#2}
    %%\node[line width=1pt,inner sep = 1pt, font=\scriptsize] (L#2) at ($(#2) + (-90+#4:0.3*\parh)$) {#3};

}
\newcommand{\addRetiLeaf}[5]{
    %% add a leaf, (#2), labelled by #3 from the node (#1) at angle #4
    %% #1 .. label of A node
    %%\node[smallNode] (#2) at ($(#1) + (-90+#4:0.5*\parh)$) {};
    \node[line width=1pt,inner sep = 1pt, font=\scriptsize] (#2) at ($(#1) + (-90+#4:#5)$) {#3};
    \doubleEdge{#1}{#2}
}
\begin{subfigure}[b]{0.44\textwidth}
    \centering
    \begin{tikzpicture}
        %% nodes
        % A
        \node[smallNode] (a1) at (0,0) {};
        \node[smallNode] (a2) at ($(a1) + (-90+0*\myTheta:0.5*\parh)$) {};
        \node[smallNode] (a3) at ($(a2) + (-90+1*\myTheta:1.0*\parh)$) {};
        \node[smallNode] (a4) at ($(a3) + (-90+1*\myTheta:1.0*\parh)$) {};
        \node[smallNode] (a5) at ($(a3) + (-90-1*\myTheta:2.0*\parh)$) {};
        \node[smallNode] (a6) at ($(a4) + (-90-1*\myTheta:1.0*\parh)$) {};
        % B
        \node[smallNode] (b1) at ($(a5) + (-90-1*\myTheta:1.0*\parh)$) {};
        \node[smallNode] (b2) at ($(b1) + (-90+0*\myTheta:0.5*\parh)$) {};
        \node[smallNode] (b3) at ($(b2) + (-90-1*\myTheta:1.0*\parh)$) {};
        % C
        \node[smallNode] (c1) at ($(a5) + (-90+1*\myTheta:1.0*\parh)$) {};
        \node[smallNode] (c2) at ($(c1) + (-90+0*\myTheta:0.5*\parh)$) {};
        % D
        \node[smallNode] (d1) at ($(a6) + (-90+1*\myTheta:1.0*\parh)$) {};
        \node[smallNode] (d2) at ($(d1) + (-90+0*\myTheta:0.5*\parh)$) {};
        \node[smallNode] (d3) at ($(d2) + (-90-1*\myTheta:1.0*\parh)$) {};
        \node[smallNode] (d4) at ($(d3) + (-90+1*\myTheta:1.0*\parh)$) {};
        \node[smallNode] (d5) at ($(d2) + (-90+1*\myTheta:2.0*\parh)$) {};
        % E
        \node[smallNode] (e1) at ($(b3) + (-90+1*\myTheta:1.0*\parh)$) {};
        \node[smallNode] (e2) at ($(e1) + (-90+0*\myTheta:0.5*\parh)$) {};
        \node[smallNode] (e3) at ($(e2) + (-90+1*\myTheta:1.0*\parh)$) {};
        \node[smallNode] (e4) at ($(e3) + (-90-1*\myTheta:1.0*\parh)$) {};
        % F
        \node[smallNode] (f1) at ($(d4) + (-90-1*\myTheta:1.0*\parh)$) {};
        % G
        \node[smallNode] (g1) at ($(d4) + (-90+1*\myTheta:1.0*\parh)$) {};
        \node[smallNode] (g2) at ($(g1) + (-90+0*\myTheta:0.5*\parh)$) {};
        \node[smallNode] (g3) at ($(g2) + (-90-1*\myTheta:1.0*\parh)$) {};
        \node[smallNode] (g4) at ($(g2) + (-90+1*\myTheta:1.0*\parh)$) {};
        % H
        \node[smallNode] (h1) at ($(e4) + (-90-1*\myTheta:1.0*\parh)$) {};
        % I
        \node[smallNode] (i1) at ($(e4) + (-90+1*\myTheta:1.0*\parh)$) {};
        \node[smallNode] (i2) at ($(i1) + (-90+0*\myTheta:0.5*\parh)$) {};
        \node[smallNode] (i3) at ($(i2) + (-90+1*\myTheta:1.0*\parh)$) {};
        % J
        \node[smallNode] (j1) at ($(g3) + (-90+1*\myTheta:1.0*\parh)$) {};

        %% edges
        % A
        \singleEdge{a1}{a2}
        \singleEdge{a2}{a3}
        \singleEdge{a3}{a4}
        \singleEdge{a3}{a5}
        \singleEdge{a4}{a6}
        % B
        \bendEdge{\myTheta}{a2}{b1}
        \singleEdge{a5}{b1}
        \singleEdge{b1}{b2}
        \singleEdge{b2}{b3}
        % C
        \singleEdge{a5}{c1}
        \singleEdge{a6}{c1}
        \singleEdge{c1}{c2}
        % D
        \bendEdge{-\myTheta}{a4}{d1}
        \singleEdge{a6}{d1}
        \singleEdge{d1}{d2}
        \singleEdge{d2}{d3}
        \singleEdge{d3}{d4}
        \singleEdge{d2}{d5}
        % E
        \bendEdge{-\myTheta}{b2}{e1}
        \singleEdge{b3}{e1}
        \singleEdge{e1}{e2}
        \singleEdge{e2}{e3}
        \singleEdge{e3}{e4}
        % F
        \bendEdge{\myTheta}{d3}{f1}
        \singleEdge{d4}{f1}
        % G
        \singleEdge{d4}{g1}
        \singleEdge{d5}{g1}
        \singleEdge{g1}{g2}
        \singleEdge{g2}{g3}
        \singleEdge{g2}{g4}
        % H
        \bendEdge{\myTheta}{e2}{h1}
        \singleEdge{e4}{h1}
        % I
        \bendEdge{-\myTheta}{e3}{i1}
        \singleEdge{e4}{i1}
        \singleEdge{i1}{i2}
        \singleEdge{i2}{i3}
        % J
        \singleEdge{g3}{j1}
        \singleEdge{g4}{j1}

        %% leaves
        \foreach \x/\t/\n in {
            h1/0/1,
            i3/-1/2,
            c2/-1/3,
            g3/-1/4,
            i2/-1/5,
            j1/0/6,
            b3/-1/7,
            c2/0/8,
            d5/1/9,
            f1/0/10,
            g4/1/11,
            i3/1/12}{
            \addLeaf{\x}{L\n}{\n}{\t*\myTheta}{0.7*\parh}
        }

        %% background
        \begin{scope}[on background layer={color=yellow}]
            \foreach \x in {
                a1, a2, a3, a4, a5, a6,
                b1, b2, b3,
                c1, c2,
                d1, d2, d3, d4, d5,
                e1, e2, e3, e4,
                g1, g2, g3, g4,
                i1, i2, i3}{ %% f1, h1, j1
                \node[circle, fill=white!80!red, minimum size=12pt] at (\x) {};
            }
            \foreach \x/\y in {
                a1/a2,a2/a3,a3/a4,a3/a5,a4/a6,
                b1/b2,b2/b3,
                c1/c2,
                d1/d2,d2/d3,d3/d4,d2/d5,
                e1/e2,e2/e3,e3/e4,
                g1/g2,g2/g3,g2/g4,
                i1/i2,i2/i3}{
                \draw[draw=white!80!red, line width=12pt] (\x.center) -- (\y.center);
            }
        \end{scope}

    \end{tikzpicture}
    \caption*{$N$}
\end{subfigure}
\begin{subfigure}[b]{0.44\textwidth}
    \centering
    \newcommand{\deltaH}{2.4}
    \begin{tikzpicture}
        %% ndoes
        \node[styNode] (A) at (0,0) {};
        \node[styNode] (B) at ($(A) + (-90-1.2*\myTheta:1.0*\parH)$) {};
        \node[styNode] (C) at ($(A) + (-90         :1.0*\parH)$) {};
        \node[styNode] (D) at ($(A) + (-90+1.2*\myTheta:1.0*\parH)$) {};
        \node[styNode] (E) at ($(B) + (-90         :1.0*\parH)$) {};
        %%\node[styNode] (F) at ($(D) + (-90         :1.0*\parH)$) {F};
        \node[styNode] (G) at ($(D) + (-90+\myTheta:1.0*\parH)$) {};
        %%\node[styNode] (H) at ($(E) + (-90-\myTheta:1.0*\parH)$) {H};
        \node[styNode] (I) at ($(E) + (-90+\myTheta:1.0*\parH)$) {};
        %%\node[styNode] (J) at ($(G) + (-90         :1.0*\parH)$) {J};

        %% to controll vertically aligned \newcommand{\deltaH}{2.4}
        \node[] (TOp) at ($(A) + ( 90:\deltaH)$) {};
        \node[] (BTM) at ($(I) + (-90:\deltaH)$) {};

        %% edges
        \doubleEdge{A}{B}
        \doubleEdge{A}{C}
        \doubleEdge{A}{D}
        \doubleEdge{B}{E}
        %%\doubleEdge{D}{F}
        \doubleEdge{D}{G}
        %%\doubleEdge{E}{H}
        \doubleEdge{E}{I}
        %%\doubleEdge{G}{J}

        %% leaves
        %%\addLeaf{H}{L1}{1}{ 0*\myTheta}
        \addLeaf{I}{L2}{2}{-1*\myTheta}{\parh}
        \addLeaf{C}{L3}{3}{-1*\myTheta}{\parh}
        \addLeaf{G}{L4}{4}{-1*\myTheta}{\parh}
        \addLeaf{I}{L5}{5}{ 0*\myTheta}{\parh}
        %%\addLeaf{J}{L6}{6}{ 0*\myTheta}
        \addLeaf{B}{L7}{7}{-1*\myTheta}{\parh}
        \addLeaf{C}{L8}{8}{ 1*\myTheta}{\parh}
        \addLeaf{D}{L9}{9}{-1*\myTheta}{\parh}
        %%\addLeaf{F}{L10}{10}{ 0*\myTheta}
        \addLeaf{G}{L11}{11}{ 1*\myTheta}{\parh}
        \addLeaf{I}{L12}{12}{ 1*\myTheta}{\parh}

        %% Reticulation Leaves
        \addRetiLeaf{D}{L10}{10}{ 0*\myTheta}{\parh}
        \addRetiLeaf{G}{L6}{6}{ 0*\myTheta}{\parh}
        \addRetiLeaf{E}{L1}{1}{-1*\myTheta}{\parh}
    \end{tikzpicture}
    \caption*{$C(N)$}
\end{subfigure}
\vspace*{0.15cm}\caption{Left: A galled network with $12$ leaves and $9$ reticulation vertices (with the tree-components highlighted). Right: The component graph of the network. Note that it is a phylogenetic tree.}
\label{fig:example_GN}
\end{figure}
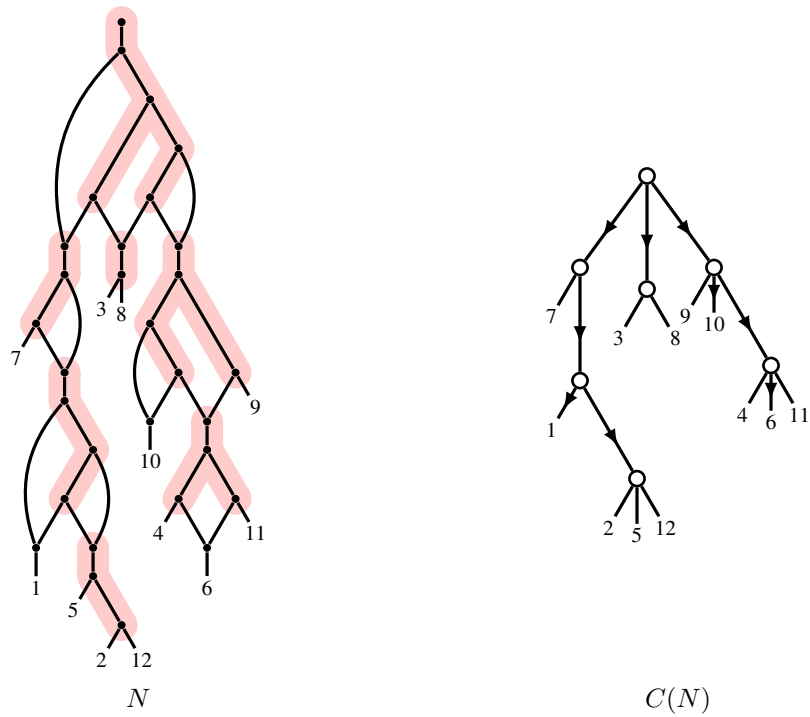

\begin{Rem}
Note that different definitions of the component graph are used for different network classes. For example, for tree-child networks, the leaves of the tree-components are not attached to the nodes of the component graph but the nodes themselves are labeled; see, e.g., \cite{CaZh,ChFuLiWaYu2,FuHuYu}. Our above definition is most suitable for galled networks and reticulation-visible networks which are the network classes considered in this paper.
\end{Rem}

\vspace*{0.25cm}If $N$ is a network, we denote its component graph by $C(N)$. The component graph can be seen as a compression of the network. We also use the notation $\tilde{C}(N)$ to denote the component graph with the arrows on single edges removed.

Note that the phylogenetic network in Figure~\ref{fig:example_RV} is a reticulation-visible network and its component graph is a (non-binary) tree-child network. This, in fact, is a general phenomenon.

\begin{thm*}[\cite{GuYaZh}]
Let $N$ be a phylogenetic network.
\begin{itemize}
\item[(i)] $N$ is galled if and only if $\tilde{C}(N)$ is a (not necessarily binary) phylogenetic tree.
\item[(ii)] $N$ is reticulation-visible if and only if $\tilde{C}(N)$ is a (not necessarily binary) tree-child network with all vertices of indegree at most $2$ and no reticulation vertex has just one child that is, moreover, a tree vertex.
\end{itemize}
\end{thm*}

See Figure~\ref{fig:example_GN} for an illustration of part (i) of the above result.

In order to generate networks, we start from the component graphs and decompress them, i.e., we generate all networks whose component graph is the given component graph. For this, the notion of {\it one-component phylogenetic networks} is useful.

\begin{Def}
A phylogenetic network is called a {\it one-component phylogenetic network} if every reticulation vertex is directly followed by a leaf.
\end{Def}

In other words, a network is one-component if it has only one non-trivial tree-component.

One-component galled networks were counted in \cite{GuRaZh}. (Note that, in fact, the class of  one-component networks and the class of one-component galled networks coincide; see \cite{CaZh}.) Denote by ${\rm OGN}_{\ell,k}$ the number of one-component galled networks with $\ell$ leaves and $k$ reticulation vertices and by $M_{\ell,k}$ those one-component galled networks whose leaves below the reticulation vertices are labeled by $\{1,\ldots,k\}$. Then, the following result was proved in \cite{GuRaZh}. (Note that part (i) is trivial.)

\begin{pro*}[\cite{GuRaZh}]
\begin{itemize}
\item[(i)] We have,
\[
{\rm OGN}_{\ell,k}=\binom{\ell}{k}M_{\ell,k}.
\]
\item[(ii)] For $2\leq k\leq\ell$,
\begin{align}
M_{\ell,k}=(\ell+k-2)&M_{\ell,k-1}+(k-1)M_{\ell,k-2}\nonumber\\
&+\frac{1}{2}\sum_{1\leq d\leq k-1}\binom{k-1}{d}(2d-1)!!\left(M_{\ell-d,k-1-d}-M_{\ell+1-d,k-1-d}\right)\label{rec-Mlk}
\end{align}
with initial values $M_{\ell,0}=(2\ell-3)!!$ and $M_{\ell,1}=(\ell-1)(2\ell-3)!!$.
%% note that (-3)!!=(-1)!!:=0. by YS 2023-12-27 17:14
\end{itemize}
\end{pro*}

\begin{figure}
    \vspace*{-0.45cm}\centering
    \includegraphics[width=0.9\textwidth]{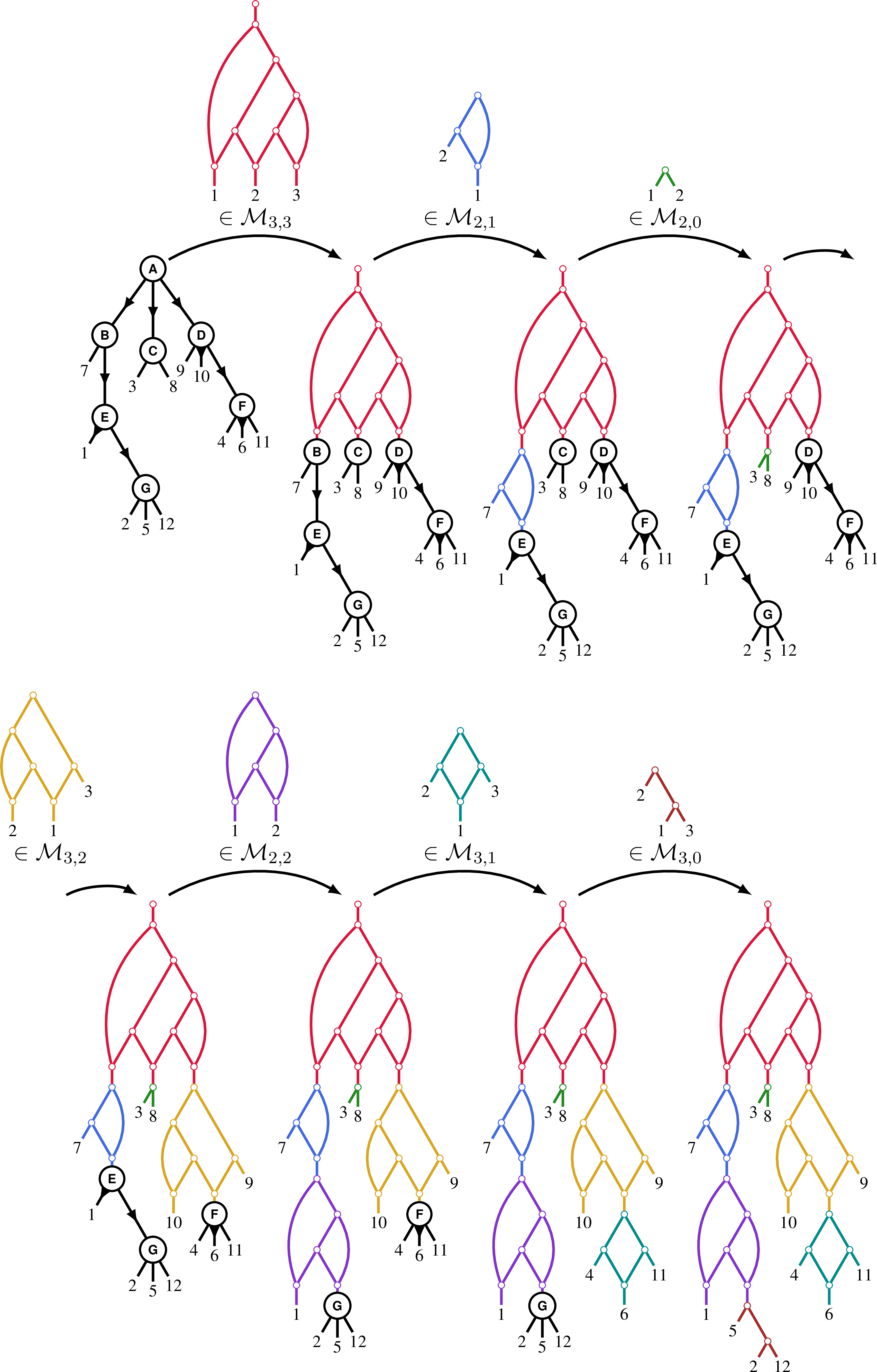}
    \caption{Step-by-step decompression of the component graph from Figure~\ref{fig:example_GN}; the nodes are processed in the order indicated and the one-component networks replacing nodes are above the arrows. Note that except for the first of these networks, the root edge has to be removed in all others.}\label{decomp-gn}
\end{figure}

\vspace*{0.2cm}Now, consider the component graph of a galled network; see Figure~\ref{decomp-gn} for an example. In order to decompress it, first pick a one-component galled network $N^{o}$ for the root $r$ of the component graph, where the number of leaves, say $\ell_r$ (the network over the first arrow in the example), of $N^{o}$ is the outdegree of $r$ and the number of reticulation vertices, say $k_r$, is the number of arrows on the outgoing edges of $r$. ($\ell_r=k_r=3$ in the example.) Moreover, let the leaves below the reticulation vertices in $N^{o}$ have labels $\{1,\ldots,k_r\}$. Next, remove $r$ from the component graph which gives a forest of $\ell_r$ trees. (The subtrees rooted at $B, C,$ and $D$ in the example.) The trees which have been attached by edges with arrows to $r$ (all trees in the example) go to the leaves below reticulation vertices in $N^{o}$ and the remaining trees (which consist of just one labeled vertex) are used the relabel the remaining leaves of $N^{o}$ in an order consistent way (i.e., the smallest goes to $k_r+1$, the second smallest to $k_r+2$, etc.). Moreover, for the trees which go to the leaves below reticulation vertices, the order in which they are attached is the increasing order of their smallest leaf label (i.e., the one with the smallest leaf label --- the subtree rooted at $B$ in the example as the smallest label of this subtree is $1$ --- goes to $1$, the one with the second smallest leaf label --- the subtree rooted at $C$ in the example as the smallest label of this subtree is $3$ --- goes to $2$, etc.). Then, continue with the non-leaf vertices of the trees in a recursive way. This gives, by picking all possible one-component networks in each step, all possible galled networks whose component graph is the given component graph.  Moreover, if the given component graph has no arrows (which for galled networks means it is just a phylogenetic tree), then arrows have to be placed on edges leading to non-leaf vertices, whereas on the pendant edges, we can freely decide whether we want to place an arrow or not. Overall, this procedure gives the following result from \cite{GuRaZh} for the number of galled networks with $\ell$ leaves, denoted by ${\rm GN}_{\ell}$.

\begin{thm*}[\cite{GuRaZh}]
For the number ${\rm GN}_{\ell}$ of galled networks with $\ell$ leaves,
\begin{equation}\label{form-GNl}
{\rm GN}_{\ell}=\sum_{{\mathcal T}}\prod_{v}\sum_{j=0}^{c_{\rm lf}(v)}\binom{c_{\rm lf}(v)}{j}M_{c(v),c_{\rm nlf}(v)+j},
\end{equation}
where the first sum runs over all (not necessarily binary) phylogenetic trees ${\mathcal T}$, the product runs over all internal vertices $v$ of ${\mathcal T}$, $c(v)$ is the outdegree of $v$, and $c_{{\rm lf}}(v)$ resp. $c_{{\rm nlf}}(v)$ are the
number of leaves resp. non-leaves amongst the children of $v$.
\end{thm*}

By keeping track of the number of arrows, ${\rm GN}_{\ell,k}$ can be computed as well.

Next, a similar procedure can be used for reticulation-visible networks, too. First, observe that the classes of one-component reticulation-visible networks and one-component galled networks coincide (since both are, in fact, the class of one-component networks without any restrictions). Thus, we can again work with $M_{\ell,k}$ which satisfies (\ref{rec-Mlk}).

Now, consider the component graph of a reticulation-visible network (which by the above result is a particular tree-child network; also, assume that there are arrows on the edges); see Figure~\ref{decomp-rv} for an example. The decompressing works as for a galled networks, with the only difference being how the children of a node are attached to the leaves of the one-component network which replaces the node. More precisely, in which order should this be done? (Again, the children attached with edges having arrows go to the leaves below reticulation vertices, the others not.) Here, the tree-child property comes into play. It implies that that for every vertex, there exists a set of leaves which can only be reached from this vertex. All children of a node have such a set and we can order these sets (and consequently the children) according to their smallest elements. This is then used to attach the children of a node to the leaves of the one-component network which replaces the node. Overall, this gives the following result for the number of reticulation-visible networks with $\ell$ leaves, denoted by ${\rm RV}_{\ell}$, where we again start from component graphs without arrows.

\begin{figure}
    \centering
    \includegraphics[width=0.9\textwidth]{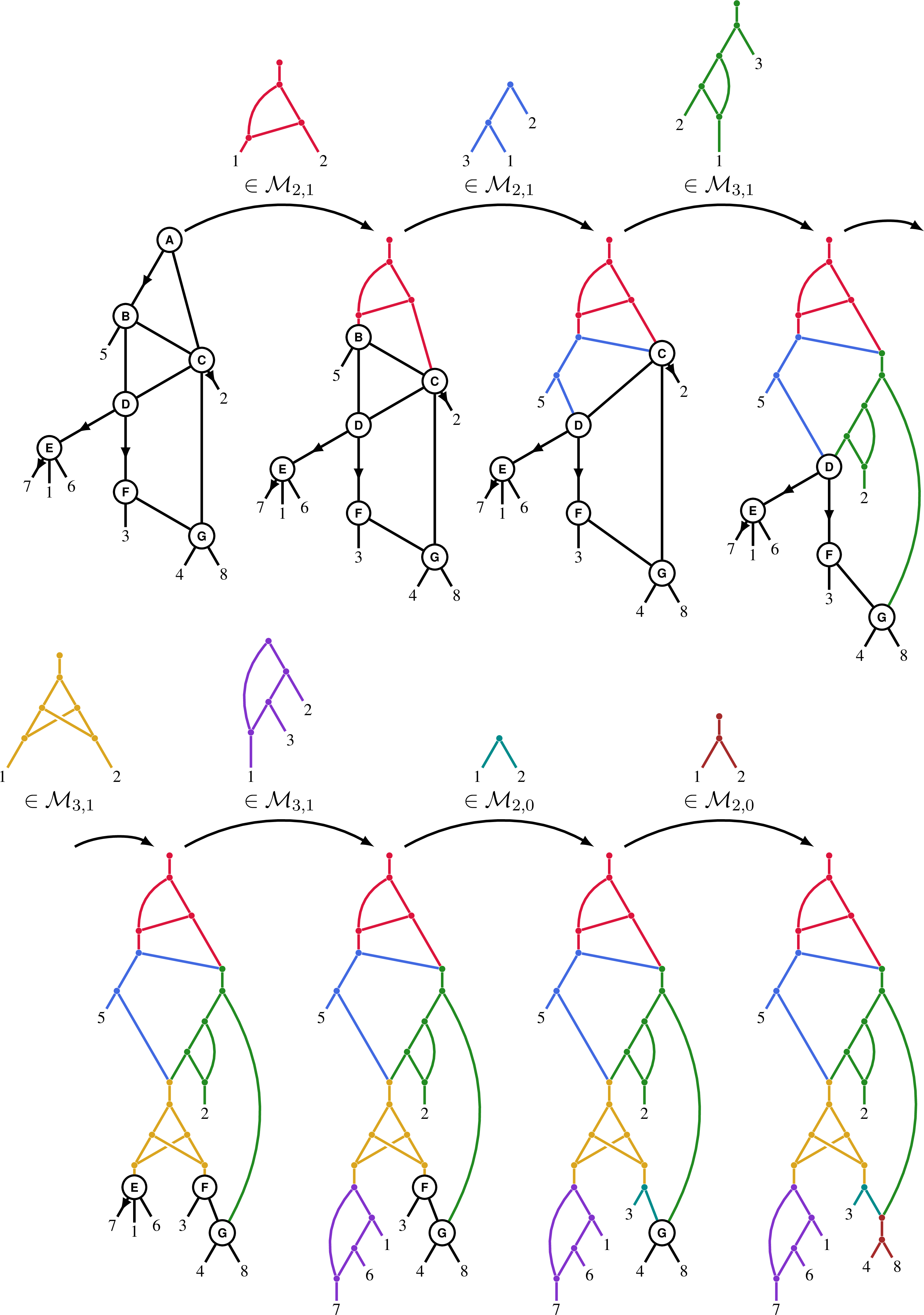}
    \vspace*{0.25cm}\caption{Step-by-step decompression of the component graph from Figure~\ref{fig:example_RV}; again the nodes are processed in the indicated order and the one-component networks are above the arrows. Here, the root edges of the latter networks have to be removed if and only if the replaced nodes in the component graph have indegree $1$ (and thus the incoming edge has an arrow on it).}\label{decomp-rv}
\end{figure}

\begin{thm}
For the number ${\rm RV}_{\ell}$ of reticulation-visible networks with $\ell$ leaves,
\begin{equation}\label{form-RVl}
{\rm RV}_{\ell}=\sum_{{\mathcal{TC}}}\prod_{v}\sum_{j=0}^{c_{\rm lf}(v)}\binom{c_{\rm lf}(v)}{j}M_{c(v),c_{1}(v)+j},
\end{equation}
where the first sum runs over all (not necessarily binary) tree-child networks with all vertices of indegree at most $2$ and no reticulation vertex has just one child that is, moreover, a tree vertex, the product runs over all internal vertices $v$ of the tree-child network, $c(v)$ is the outdegree of $v$, $c_{\rm lf}(v)$ is the number of leaves below $v$, and $c_{1}(v)$ is the number of children of $v$ which are not leaves and have indegree $1$.
\end{thm}

Again, by keeping track of arrows, we can compute ${\rm RV}_{\ell,k}$. We will do this in Section~\ref{ret-vis-net} to derive formulas for small values of $k$ and any $\ell$.

\section{Galled Networks}\label{gall-net}

In this section, we use the component graph method to count (both exactly and asymptotically) galled networks with a small number of reticulation vertices. In addition, we use generating function techniques and the method of singularity analysis which is described in Chapter VI of \cite{FlSe}. (We will use below some of the notation from that chapter.)

Since the component graph method is based on one-component galled networks, we start by analyzing their number. First, from the initial values and (\ref{rec-Mlk}), $M_{\ell,k}$ for small values of $k$ equals
\begin{align*}
M_{\ell,0}&=(2\ell-3)!!,\\
M_{\ell,1}&=(\ell-1)(2\ell-3)!!,\\
M_{\ell,2}&=(2\ell-1)(\ell-1)^2(2\ell-5)!!.
\end{align*}

From this, we observe the following pattern.
\begin{lmm}
For fixed $k\geq 1$,
\[
M_{\ell,k}=p_k(\ell)(2(\ell-k)-3)!!,\qquad (\ell\geq k),
\]
where $p_k(\ell)$ is a polynomial of degree $2k$ with leading coefficient $2^k$.
\end{lmm}
\begin{proof}
This follows by induction on $k$. First, the claim holds for $k=1, k=2$, and also for $k=0$ when $\ell\geq 1$. (This is why we used $(2(\ell-k)-3)!!$ instead of $(2(\ell-k+1)-3)!!$.)

Assume now that it holds for $2\leq k'<k$. Then, we observe that there are four terms on the right-hand side of (\ref{rec-Mlk}) where we have to plug in the induction hypothesis. After doing this and bringing the double factorial into the form $(2(\ell-k)-3)!!$, we see that the first term ($M_{\ell,k-1}$) has as multiplicative factor a polynomial in $\ell$ of degree $2k-1$ which gets multiplied by $\ell+k-2$ and thus becomes a polynomial of degree $2k$, the second term ($M_{\ell,k-2}$) produces a polynomial of degree $2k-2$, the third factor ($M_{\ell-d,k-1-d}$) yields a polynomial of degree $2k-1-2d$, and the polynomial of the final factor ($M_{\ell+1-d,k-1-d}$) is of degree $2k-2d$. Thus, by collecting all these polynomials, we see that $M_{\ell,k}$ has the desired form where the degree and leading term comes from that of the first term. The latter is the leading term of the polynomial of $M_{\ell,k-1}$ multiplied by $2\ell^2$ which proves the claim.
\end{proof}

Note that from the last result, we have
\begin{equation}\label{form-Fkl}
M_{\ell+k,k}=q_k(\ell)(2\ell-3)!!,\qquad (\ell\geq 0),
\end{equation}
where $q_k(\ell)$ is again a polynomial of degree $2k$ with leading term $2^k$. We next consider the exponential generating function of this quantity:
\[
F_k(z)=\sum_{\ell\geq 0}M_{\ell+k,k}\frac{z^{\ell}}{\ell!},
\]
where $M_{0,0}=0$. Then, from the expressions for $M_{\ell,k}$ for $k=0,1,2$, we obtain that:
\begin{equation}\label{F0-1}
F_{0}(z)=1-\sqrt{1-2z},\qquad F_{1}(z)=\frac{z}{(1-2z)^{3/2}}
\end{equation}
and
\begin{equation}\label{F2}
F_{2}(z)=\frac{3-z+7z^2-4z^3}{(1-2z)^{7/2}}.
\end{equation}
Also, we have the following asymptotic result for $F_{k}(z)$ for all $k\geq 1$, where we use the notion of $\Delta$-analyticity of a function $f(z)$ at $z_0$, i.e., $f(z)$ is an analytic function in a domain of the form
\[
\Delta:=\{z\ :\ \vert z\vert<r,\ \vert\arg(z-z_0)\vert>\phi\},
\]
for some $r>\vert z_0\vert$ and $0<\phi<\pi/2$.

%\begin{lmm}
%For $k\geq 1$, $M_k(z)$ is $\Delta$-analytic for some $\Delta$-domain at $1/2$ and satisfies, as $z\rightarrow 1/2$, in the $\Delta$-domain:
%\[
%M_k(z)\sim\frac{(2k-3)!!}{2^k(1-2z)^{k-1/2}}.
%\]
%\end{lmm}
%\begin{proof}
%Note that
%\[
%P_k(z):=\sum_{\ell\geq k}(2(\ell-k)-3)!!\frac{z^{\ell}}{\ell!}=\sum_{\ell\geq 0}(2\ell-3)!!\frac{z^{\ell}}{\ell!}=2-\sqrt{1-2z}.
%\]
%From Lemma~\ref{form-Mkl}, we see that the function $M_k(z)$ is built from $P_k(z)$ as a linear combination of
%\[
%D^{j}P_k(z),
%\]
%where $0\leq j\leq 2k$ and $D^{j}$ is the $j$-th iteration of $D:=z\frac{d}{dz}$. This shows that $M_k(z)$ is $\Delta$-analytic with singularity expansion, as $z\rightarrow 1/2$,
%\[
%M_k(z)\sim 2^kD^{2k}P_k(z)=\frac{(2k-3)!!}{2^k(1-2z)^{k-1/2}}
%\]
%which is the claim.
%\end{proof}

%We also need the following closely related generating function
%\[
%F_k(z)=\sum_{\ell\geq 0}M_{\ell+k,k}\frac{z^{\ell}}{\ell!}
%\]
%which satisfies the following result.

\begin{lmm}
For $k\geq 1$, $F_k(z)$ is $\Delta$-analytic for some $\Delta$-domain at $1/2$ and satisfies, as $z\rightarrow 1/2$, in the $\Delta$-domain:
\begin{equation}\label{asym-Fkz}
F_k(z)\sim\frac{(4k-3)!!}{2^k(1-2z)^{2k-1/2}}.
\end{equation}
\end{lmm}

\begin{Rem}
The asymptotic (\ref{asym-Fkz}) is called {\it singularity expansion} of $F_k(z)$ at $z\rightarrow 1/2$. Note that $k=0$ is not included since from (\ref{F0-1}), $F_0(z)\sim 1$, as $z\rightarrow 1/2$, which does not follow the pattern from (\ref{asym-Fkz}).
\end{Rem}

\vspace*{-0.35cm}
\begin{proof}
First, set
\begin{equation}\label{Pz}
P(z):=\sum_{\ell\geq 0}(2\ell-3)!!\frac{z^{\ell}}{\ell!}=2-\sqrt{1-2z}.
\end{equation}
From (\ref{form-Fkl}), we see that the function $F_k(z)$ is built from $P(z)$ as a linear combination of the power series
\[
D^{j}P(z),
\]
where $0\leq j\leq 2k$ and $D^{j}$ is the $j$-th iteration of $D:=z\frac{d}{dz}$. Since $P(z)$ is clearly $\Delta$-analytic, and $\Delta$-analyticity is closed under differentiation (see Theorem VI.8 in \cite{FlSe}), multiplication by $z$, and linear combination (see Section VI.6 in \cite{FlSe}), we obtain that $F_k(z)$ is $\Delta$-analytic. In addition, the singularity expansion of the derivative is obtained by differentiating the singularity expansion of a function (again see Theorem VI.8 in \cite{FlSe}) and obvious rules hold for multiplying with $z$ (which introduces a factor of $1/2$, as $z\rightarrow 1/2$) and taking linear combinations (see again Section VI.6 in \cite{FlSe}). Thus, the dominant term in the singularity expansion expansion of $F_k(z)$ arises from the highest derivative, and so we have, as $z\rightarrow 1/2$,
\[
F_k(z)\sim 2^kD^{2k}P(z)\sim\frac{(4k-3)!!}{2^k(1-2z)^{2k-1/2}},
\]
where the second asymptotic equivalence follows by differentiating (\ref{Pz}) $2k$ times, multiplying with $1/2$ after every differentiation, and multiplying the final result by $2^k.$
\end{proof}

%\begin{proof}
%Note that
%\begin{equation}\label{Fk}
%F_k(z)=\frac{d^k}{dz^k}M_k(z).
%\end{equation}
%Thus, the result follows from the closure properties of singularity analysis; see Section VI.6 in \cite{FlSe}
%\end{proof}
%\begin{Rem}\label{exp-Fk}
%Note that $k=0$ has a different singularity expansion since $F_0(z)=1-\sqrt{1-2z}$. Also, from the above results for $M_k(z)$ and (\ref{Fk}):
%\[
%F_1(z)=\frac{z}{(1-2z)^{3/2}},\qquad F_2(z)=\frac{3-z+7z^2-4z^3}{(1-2z)^{7/2}}.
%\]
%\end{Rem}

We now consider general galled networks which are built from one-component galled networks and component graphs. This was described in detail in Section~\ref{cgm} and the recursive method there can be translated into generating functions (because the component graphs are trees). Since we are interested in reticulation vertices, we need to keep track of them. We therefore consider the generating function
\[
G(z,v):=\sum_{\ell\geq 0}\sum_{k\geq 0}{\rm GN}_{\ell,k}\frac{z^{\ell}}{\ell!}v^{k}.
\]
Then, we have the following result.
\begin{pro}\label{prop-Gzv}
We have,
\begin{equation}\label{Gzv}
G(z,v)=\sum_{j\geq 0}F_j(z)\frac{(vG(z,v))^j}{j!}.
\end{equation}
\end{pro}
\begin{proof}
We use symbolic combinatorics as described in Chapter II and Chapter III of \cite{FlSe}. Note that the decompression procedure of component graphs from Section~\ref{cgm} entails that every galled network is built from a one-component galled network (which was used to replace the root in the decompression procedure) whose leaves below reticulation vertices are replaced by an unordered sequence of galled networks. If the former has $j$ reticulation vertices, it is counted by $F_j(z)$ and the latter is then counted by $(vG(z,v))^j/j!$ where $v$ counts reticulation vertices, $G(z,v)^j$ counts ordered sequences of galled networks and the factor $1/j!$ is needed to take away the order. Next, the product of these two generating functions counts the galled networks which are re-labelled as described in Section~\ref{cgm} since the product of exponential generating functions corresponds to the product of labeled combinatorial classes; see Chapter II in \cite{FlSe}. Finally, summing over $j$ gives the claimed result.
\end{proof}

The exponential generating function for the number of galled networks with $k$ reticulation vertices, i.e.,
\[
E_k(z):=\sum_{\ell\geq k}{\rm GN}_{\ell,k}\frac{z^{\ell}}{\ell!}
\]
is obtained from $G(z,v)$ by partial differentiation and evaluating at $v=0$:
\[
E_k(z)=\frac{1}{k!}\frac{\partial^k}{\partial v^k}G(z,v)\Big\vert_{v=0}.
\]
From Proposition~\ref{prop-Gzv}, we obtain a recurrence.
\begin{lmm}
For $k\geq 1$,
\begin{equation}\label{exp-Ekz}
E_k(z)=\sum_{j=1}^{k}\frac{F_j(z)}{j!}\sum_{\ell_1+\cdots+\ell_j=k-j}E_{\ell_1}(z)\cdots E_{\ell_j}(z).
\end{equation}
\end{lmm}
\begin{proof}
Differentiating (\ref{Gzv}) $k$-times and evaluating at $v=0$ gives
\begin{align*}
E_{k}(z)&=\frac{1}{k!}\sum_{j=1}^{k}\frac{F_j(z)}{j!}\frac{d^k}{dv^k}\left(vG(z,v)\right)^j\Big\vert_{v=0}\\
&=\frac{1}{k!}\sum_{j=1}^{k}\frac{F_j(z)}{j!}\binom{k}{j}j!\frac{d^{k-j}}{dz^{k-j}}G(z,v)^j\\
&=\sum_{j=1}^{k}\frac{F_j(z)}{j!(k-j)!}\sum_{\ell_1+\cdots+\ell_{j}=k-j}\binom{k-j}{\ell_1,\cdots,\ell_j}\ell_1!E_{\ell_1}(z)\cdots \ell_j!E_{\ell_j}(z).
\end{align*}
The claim is obtained by writing the multinomial coefficients as factorials and canceling terms.
\end{proof}

We have now everything ready to prove Theorem~\ref{exact-GN}.

\begin{proof}[Proof of Theorem~\ref{exact-GN}.]
Note that $E_0(z)=1-\sqrt{1-2z}$. Then, from (\ref{exp-Ekz}) and (\ref{F0-1}):
\[
E_1(z)=F_1(z)E_0(z)=\frac{z(1-\sqrt{1-2z})}{(1-2z)^{3/2}}.
\]
Moreover, by using once again (\ref{exp-Ekz}) and (\ref{F0-1}) as well as (\ref{F2}):
\begin{align*}
E_2(z)&=F_1(z)E_1(z)+\frac{F_2(z)E_0(z)^2}{2}\\
&=\frac{12z^4-18z^3+17z^2-36z+21+(12z^3-10z^2+15z-21)\sqrt{1-2z}}{3(1-2z)^{7/2}}.
\end{align*}
Extracting coefficients gives the claimed result for ${\rm GN}_{\ell,2}$.

As for ${\rm GN}_{\ell,3}$, the same method can be used, only the resulting computation is more tedious (and therefore best done with mathematical software, e.g., Maple).
\end{proof}

What is left is to prove the asymptotic counting result for ${\rm GN}_{\ell,k}$ from Theorem~\ref{asymp-GN}. This will follow from the following result for the singularity expansion of $E_k(z)$ which is obtained from (\ref{exp-Ekz}) and induction.
\begin{pro}\label{asymp-Ekz}
For $k\geq 1$, $E_k(z)$ is $\Delta$-analytic for some $\Delta$-domain at $1/2$ and satisfies, as $z\rightarrow 1/2$, in the $\Delta$-domain:
\[
E_k(z)\sim\frac{(4k-3)!!}{k!2^k(1-2z)^{2k-1/2}}.
\]
\end{pro}
\begin{proof}
Note that
\[
E_0(z)=F_0(z)=1-\sqrt{1-2z}\sim 1.
\]
Thus, if $k=0$ is included in the claim, then the power of $1-2z$ in the denominator is $\max\{2k-1/2,0\}$.

We use induction on $k$. For $k=1$, the claim holds since $E_1(z)\sim F_1(z)$ which has the desired form by (\ref{asym-Fkz}). Thus, we can assume that the claim holds for $k'<k$. We need to prove it for $k$. Plugging the induction hypothesis into (\ref{exp-Ekz}) and using (\ref{asym-Fkz}), we obtain that for the terms inside the double sum of (\ref{exp-Ekz}):
\[
F_j(z)E_{\ell_1}(z)\cdots E_{\ell_j}(z)\sim c(1-2z)^{-(2j-1/2+\max\{2\ell_1-1/2,0\}+\cdots+\max\{2\ell_j-1/2,0\})},
\]
where $c$ is a suitable constant and $\ell_1+\cdots+\ell_j=k-j$. The term inside the bracket is maximized if and only if $j=k$ and thus $\ell_1=\cdots=\ell_{k}=0$. This shows that
\[
E_k(z)\sim\frac{F_k(z)}{k!}
\]
which by (\ref{asym-Fkz}) gives the claim.
\end{proof}

Now, we can prove Theorem~\ref{asymp-GN} for $k\geq 1$.
\begin{proof}[Proof of Theorem~\ref{asymp-GN}.]
From Proposition~\ref{asymp-Ekz} and Corollary VI.1 in \cite{FlSe}:
\begin{align*}
{\rm GN}_{\ell,k}=\ell![z^{\ell}]E_k(z)&\sim \ell!\frac{(4k-3)!!}{k!2^k\Gamma(2k-1/2)}[z^n](1-2z)^{-2k+1/2}\\
&\sim \ell!\frac{(4k-3)!!}{k!2^k\Gamma(2k-1/2)}2^{\ell} \ell^{2k-3/2}.
\end{align*}
Note that
\[
\Gamma(2k-1/2)=2^{-2k+1}(4k-3)!!\sqrt{\pi}.
\]
Plugging this into the above expression and using Stirling's formula gives the claimed result.
\end{proof}

\begin{Rem}
It is easily verified that Theorem~\ref{asymp-GN} holds for $k=0$, too.
\end{Rem}

\section{Reticulation-Visible Networks}\label{ret-vis-net}

In this section, we prove the formulas for the number of reticulation-visible networks with $\ell$ leaves and $k$ reticulation vertices with $k=2$ and $k=3$ to establish Theorem~\ref{exact-RV}.

\newpage

Let $N$ be a reticulation-visible network. We remove the leaves and their pendant edges from $C(N)$ except those whose pendant edge has an arrow on it. Then, the resulting component graph is a (unlabeled) rooted simple DAG with every non-root vertex of indegree $2$. (Note that an edge with an arrow on it is actually a double edge and thus counts as two edges.) Moreover, this DAG has exactly $k+1$ vertices. For $k=2$ and $k=3$, all these DAGs are listed in Figure~8 in \cite{CaZh}; see also Figure~\ref{fig:compGraphOfRVk2} and Figure~\ref{fig:compGraphOfRVk3} below.

Using these DAGs and the decompression procedure explained in Section~\ref{cgm}, we can derive the exponential generating function of ${\rm RV}_{\ell,k}$.

\begin{pro}\label{exp-RV}
Let $\mathcal{D}_{m}$ be the set of (unlabelled) rooted DAGs with $m$ vertices in which non-root vertices have indegree $2$ and double edges between two vertices are allowed. Then, for given $k$,
\begin{equation}
\sum_{\ell\geq 1}\,{\rm RV}_{\ell,k}\,\frac{z^\ell}{\ell!}=\sum_{G\in \mathcal{D}_{k+1}}\,\frac{1}{{\rm m}(G)}\prod_{v \in G}\,\sum_{\ell \geq \ell_0}\,M_{\ell+c(v),c_1(v)}\,\frac{z^\ell}{\ell!},
\end{equation}
where ${\rm m}(G)$ counts symmetries in $G$, the product runs over all vertices in $G$, and the final sum on the right-hand side is the generating function with respect to the number of labeled leaves which are attached to $v$, where $c(v)$ is the outdegree of $v$ in $G$, $c_1(v)$ is the number of children of $v$ with arrow on their edges, and $\ell_0$ is $0$ or $1$ according to whether $c_1(v)>0$ or $c_1(v)=0$, respectively.
\end{pro}

\tikzset{
    styNode/.style={
        circle, draw, line width=1pt, inner sep = 2pt, font=\scriptsize
    },
}

\begin{figure}[ht]
    \newcommand{\doubleEdge}[2]{
        %%\draw[-{Latex[length=3mm]}, line width=3pt] (#1) -- (#2);
        %%\draw[white, line width=1pt, shorten >= 3mm] (#1) -- (#2);
        \draw[line width=1.5pt, decoration={markings,mark=at position 0.6 with {\arrow{latex[length=3mm]}}}, postaction={decorate}] (#1) -- (#2);
    }
    \newcommand{\singleEdge}[2]{
        %%\draw[-{Latex[length=3mm]}, line width=1.5pt] (#1) -- (#2);
        \draw[line width=1.5pt] (#1) -- (#2);
    }
    \centering
    \begin{subfigure}[b]{0.3\textwidth}
        \centering
        \begin{tikzpicture}
            \node[styNode,label=right:$a_1$] (A1) at (  0   ,  0) {};
            \node[styNode,label=right:$a_2$] (A2) at ( -1.15, -2) {};
            \node[styNode,label=right:$a_3$] (A3) at (  1.15, -2) {};

            \doubleEdge{A1}{A2}
            \doubleEdge{A1}{A3}
        \end{tikzpicture}
        \caption*{$A$}
    \end{subfigure}
    \begin{subfigure}[b]{0.3\linewidth}
        \centering
        \begin{tikzpicture}
            \node[styNode,label=right:$b_1$] (B1) at ( 0,  0) {};
            \node[styNode,label=right:$b_2$] (B2) at ( 0, -2) {};
            \node[styNode,label=right:$b_3$] (B3) at ( 0, -4) {};

            \doubleEdge{B1}{B2}
            \doubleEdge{B2}{B3}
        \end{tikzpicture}
        \caption*{$B$}
    \end{subfigure}
    \begin{subfigure}[b]{0.3\linewidth}
        \centering
        \begin{tikzpicture}
            \node[styNode,label=right:$c_1$] (C1) at (  0   ,  0) {};
            \node[styNode,label=right:$c_2$] (C2) at ( -1.15, -2) {};
            \node[styNode,label=right:$c_3$] (C3) at (  0   , -4) {};

            \doubleEdge{C1}{C2}
            \singleEdge{C1}{C3}
            \singleEdge{C2}{C3}
        \end{tikzpicture}
        \caption*{C}
    \end{subfigure}
    \caption{The $3$ DAGs from the set $\mathcal{D}_3$ where for convenience, we have labeled the vertices. Call the reticulation vertices of the decompressed reticulation-visible networks $r_1$ and $r_2$. Then, $A$ gives all the networks where $r_1$ and $r_2$ are not in an ancestor-descendant relationship; $B$ gives all networks where $r_1$ is above $r_2$ but they are not in a tree-cycle; and $C$ gives all networks where $r_1$ is in the tree cycle of $r_2$.}
    \label{fig:compGraphOfRVk2}
\end{figure}
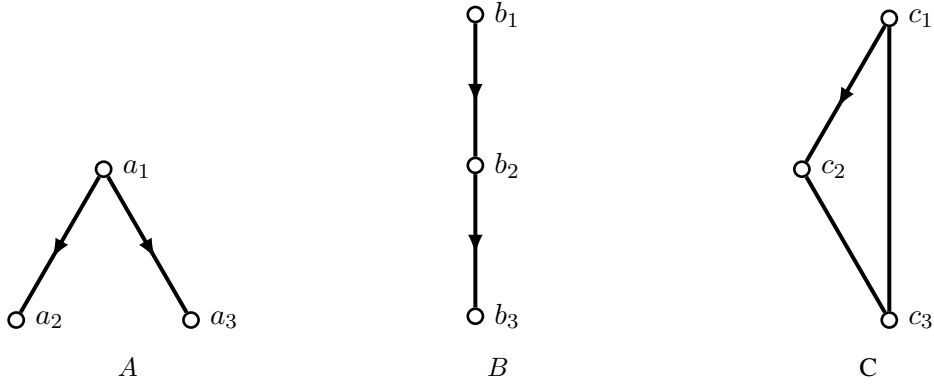

Now, to find ${\rm RV}_{\ell,2}$, we start from the DAGs in the set $\mathcal{D}_{3}$ which are listed in Figure~\ref{fig:compGraphOfRVk2}. Note that $m(A)=2$ (due to the symmetry about the root)$,m(B)=m(C)=1$, and for each vertex $v$, the exponential generating function $f_v(z)$ from Proposition~\ref{exp-RV} equals:
\begin{align*}
    f_{a_1}(z)
    & = F_2(z)
    = \sum_{\ell\geq 0} M_{\ell+2,2}\, \frac{z^n}{n!}
    = \frac{15}{4}\,X^{-7} - \frac{3}{2}\,X^{-5} + \frac{1}{4}\,X^{-3} + \frac{1}{2} \,X^{-1}; \\
    f_{a_2}(z)
    & = f_{a_3}(z) = f_{b_3}(z) = f_{c_3}(z)
    = F_0(z)
    = \sum_{\ell\geq 1} M_{\ell,0} \, \frac{z^{\ell}}{\ell!}
    = 1 - X; \\
    f_{b_1}(z)
    & = f_{b_2}(z)
    = F_1(z)
    = \sum_{\ell \geq 0} M_{\ell+1,1} \, \frac{z^{\ell}}{\ell!}
    = \frac{1}{2}\,X^{-3} - \frac{1}{2}\,X^{-1}; \\
    f_{c_1}(z)
    & = \sum_{\ell\geq 0} M_{\ell+2,1} \, \frac{z^{\ell}}{\ell!}
    = \frac{3}{2}\,X^{-5} - \frac{1}{2}\,X^{-3}; \\
    f_{c_2}(z)
    & = \sum_{\ell\geq 1} M_{\ell+1,0} \, \frac{z^{\ell}}{\ell!}
    = X^{-1}- 1,
\end{align*}
where we have used the abbreviation $X:=\sqrt{1-2z}$. Thus, from Proposition~\ref{exp-RV}, we have for the exponential generating function of ${\rm RV}_{\ell,2}$,
\begin{equation*}
    \frac{(3-z+7z^2-4z^3)(1-z-\sqrt{1-2z})}{(1-2z)^{7/2}}
    = \frac{(1-X)^2(15-6X^2+X^4+2X^6)}{8X^7}.
\end{equation*}
What is left is to extract coefficients which can be done with the following lemma.
\begin{lmm}\label{coef-Xd}
\newcommand{\ds}{\displaystyle}
We have, for $n$ large enough,
\begin{center}$$[z^n] X^{d} = \left\{\begin{array}{cl}
    \ds 0,
        & \text{if } d \geq 0 \text{ and } d \text{ is even, } k := \frac{d}{2}; \\
    \ds (-1)^{k+1}(2k+1)!!\frac{(2n-2k-3)!!}{n!},
        & \text{if } d \geq 0 \text{ and } d \text{ is odd, } k := \frac{d-1}{2}; \\
    \ds 2^n\,\binom{n+k-1}{k-1},
        & \text{if } d < 0 \text{ and } d \text{ is even, }  k := -\frac{d}{2}; \\
    \ds \frac{1}{(2k-3)!!}\frac{(2n+2k-3)!!}{n!},
        & \text{if } d < 0 \text{ and } d \text{ is odd, }  k := -\frac{d-1}{2}.\\
\end{array}\right.$$
\end{center}
\end{lmm}

\vspace*{0.2cm}\begin{proof}
All cases are obtained by the binomial theorem and standard computations.
\end{proof}

Applying the lemma gives the following result.
\[
{\rm RV}_{\ell,2}=\frac{6\ell^4+7\ell^3+6\ell^2-\ell-3}{3}(2\ell-3)!!-2^{\ell-1}(2\ell^2+2\ell+1)\ell!.
\]

\tikzset{
    styNode/.style={
        circle, draw, line width=1pt, inner sep = 2pt, font=\scriptsize
    },
}

\begin{figure}[t]
\newcommand{\doubleEdge}[2]{
        %%\draw[-{Latex[length=3mm]}, line width=3pt] (#1) -- (#2);
        %%\draw[white, line width=1pt, shorten >= 3mm] (#1) -- (#2);
        \draw[line width=1.5pt, decoration={markings,mark=at position 0.6 with {\arrow{latex[length=3mm]}}}, postaction={decorate}] (#1) -- (#2);
    }
\newcommand{\singleEdge}[2]{
    %% \draw[-{Latex[length=3mm]}, line width=1.5pt] (#1) -- (#2);
    \draw[line width=1.5pt] (#1) -- (#2);
}
\newcommand{\bendEdge}[3]{
    %% \draw[-{Latex[length=2mm]}, line width=1.5pt] (#2) to [bend right=#1] (#3);
    \draw[line width=1.5pt] (#2) to [bend right=#1] (#3);
}
\newcommand{\parH}{1.5}
\newcommand{\parh}{0.622*\parH}
\centering
\begin{subfigure}[b]{0.24\textwidth}
    \centering
    \begin{tikzpicture}
        \node[styNode] (A) at (  0,  0) {};
        \node[styNode] (B) at ($(A) + (-120:\parH)$) {};
        \node[styNode] (C) at ($(A) + (-90:\parH)$) {};
        \node[styNode] (D) at ($(A) + (-60:\parH)$) {};

        \doubleEdge{A}{B}
        \doubleEdge{A}{C}
        \doubleEdge{A}{D}
    \end{tikzpicture}
    \caption*{$A_1$}
\end{subfigure}
\begin{subfigure}[b]{0.24\textwidth}
    \centering
    \begin{tikzpicture}
        \node[styNode] (A) at (  0,  0) {};
        \node[styNode] (B) at ($(A) + (-90:\parH)$) {};
        \node[styNode] (C) at ($(B) + (-120:\parH)$) {};
        \node[styNode] (D) at ($(B) + (-60:\parH)$) {};

        \doubleEdge{A}{B}
        \doubleEdge{B}{C}
        \doubleEdge{B}{D}
    \end{tikzpicture}
    \caption*{$A_2$}
\end{subfigure}
\begin{subfigure}[b]{0.24\textwidth}
    \centering
    \begin{tikzpicture}
        \node[styNode] (A) at (  0,  0) {};
        \node[styNode] (B) at ($(A) + (-120:\parH)$) {};
        \node[styNode] (C) at ($(A) + (-60:\parH)$) {};
        \node[styNode] (D) at ($(C) + (-90:\parH)$) {};

        \doubleEdge{A}{B}
        \doubleEdge{A}{C}
        \doubleEdge{C}{D}
    \end{tikzpicture}
    \caption*{$A_3$}
\end{subfigure}
\begin{subfigure}[b]{0.24\textwidth}
    \centering
    \begin{tikzpicture}
        \node[styNode] (A) at (  0,  0) {};
        \node[styNode] (B) at ($(A) + (-90:\parh)$) {};
        \node[styNode] (C) at ($(B) + (-90:\parh)$) {};
        \node[styNode] (D) at ($(C) + (-90:\parh)$) {};

        \doubleEdge{A}{B}
        \doubleEdge{B}{C}
        \doubleEdge{C}{D}
    \end{tikzpicture}
    \caption*{$A_4$}
\end{subfigure}
\begin{subfigure}[b]{0.24\textwidth}
    \centering
    \begin{tikzpicture}
        \node[styNode] (A) at (  0,  0) {};
        \node[styNode] (B) at ($(A) + (-90:\parH)$) {};
        \node[styNode] (C) at ($(B) + (-120:\parH)$) {};
        \node[styNode] (D) at ($(B) + (-60:\parH)$) {};

        \doubleEdge{A}{B}
        \bendEdge{15}{A}{C}
        \singleEdge{B}{C}
        \bendEdge{-15}{A}{D}
        \singleEdge{B}{D}
    \end{tikzpicture}
    \caption*{$B_1$}
\end{subfigure}
\begin{subfigure}[b]{0.24\textwidth}
    \centering
    \begin{tikzpicture}
        \node[styNode] (A) at (  0,  0) {};
        \node[styNode] (B) at ($(A) + (-90:\parH)$) {};
        \node[styNode] (C) at ($(B) + (-120:\parH)$) {};
        \node[styNode] (D) at ($(B) + (-60:\parH)$) {};

        \doubleEdge{A}{B}
        \bendEdge{15}{A}{C}
        \singleEdge{B}{C}
        \doubleEdge{B}{D}
    \end{tikzpicture}
    \caption*{$B_2$}
\end{subfigure}
\begin{subfigure}[b]{0.24\textwidth}
    \centering
    \begin{tikzpicture}
        \node[styNode] (A) at (  0,  0) {};
        \node[styNode] (B) at ($(A) + (-120:\parH)$) {};
        \node[styNode] (C) at ($(A) + (-60:\parH)$) {};
        \node[styNode] (D) at ($(B) + (-60:\parH)$) {};

        \doubleEdge{A}{B}
        \doubleEdge{A}{C}
        \singleEdge{A}{D}
        \singleEdge{B}{D}
    \end{tikzpicture}
    \caption*{$B_3$}
\end{subfigure}
\begin{subfigure}[b]{0.24\textwidth}
    \centering
    \begin{tikzpicture}
        \node[styNode] (A) at (  0,  0) {};
        \node[styNode] (B) at ($(A) + (-120:\parH)$) {};
        \node[styNode] (C) at ($(A) + (-60:\parH)$) {};
        \node[styNode] (D) at ($(B) + (-60:\parH)$) {};

        \doubleEdge{A}{B}
        \doubleEdge{A}{C}
        \singleEdge{B}{D}
        \singleEdge{C}{D}
    \end{tikzpicture}
    \caption*{$B_4$}
\end{subfigure}
\begin{subfigure}[b]{0.19\textwidth}
    \centering
    \begin{tikzpicture}
        \node[styNode] (A) at (  0,  0) {};
        \node[styNode] (B) at ($(A) + (-90:\parh)$) {};
        \node[styNode] (C) at ($(B) + (-90:\parh)$) {};
        \node[styNode] (D) at ($(C) + (-90:\parh)$) {};

        \doubleEdge{A}{B}
        \bendEdge{30}{A}{C}
        \singleEdge{B}{C}
        \doubleEdge{C}{D}
    \end{tikzpicture}
    \caption*{$B_5$}
\end{subfigure}
\begin{subfigure}[b]{0.19\textwidth}
    \centering
    \begin{tikzpicture}
        \node[styNode] (A) at (  0,  0) {};
        \node[styNode] (B) at ($(A) + (-90:\parh)$) {};
        \node[styNode] (C) at ($(B) + (-90:\parh)$) {};
        \node[styNode] (D) at ($(C) + (-90:\parh)$) {};

        \doubleEdge{A}{B}
        \bendEdge{30}{A}{C}
        \singleEdge{B}{C}
        \bendEdge{-30}{B}{D}
        \singleEdge{C}{D}
    \end{tikzpicture}
    \caption*{$B_6$}
\end{subfigure}
\begin{subfigure}[b]{0.19\textwidth}
    \centering
    \begin{tikzpicture}
        \node[styNode] (A) at (  0,  0) {};
        \node[styNode] (B) at ($(A) + (-90:\parh)$) {};
        \node[styNode] (C) at ($(B) + (-90:\parh)$) {};
        \node[styNode] (D) at ($(C) + (-90:\parh)$) {};

        \doubleEdge{A}{B}
        \bendEdge{30}{A}{C}
        \singleEdge{B}{C}
        \bendEdge{-30}{A}{D}
        \singleEdge{C}{D}
    \end{tikzpicture}
    \caption*{$B_7$}
\end{subfigure}
\begin{subfigure}[b]{0.19\textwidth}
    \centering
    \begin{tikzpicture}
        \node[styNode] (A) at (  0,  0) {};
        \node[styNode] (B) at ($(A) + (-90:\parh)$) {};
        \node[styNode] (C) at ($(B) + (-90:\parh)$) {};
        \node[styNode] (D) at ($(C) + (-90:\parh)$) {};

        \doubleEdge{A}{B}
        \doubleEdge{B}{C}
        \bendEdge{30}{B}{D}
        \singleEdge{C}{D}
    \end{tikzpicture}
    \caption*{$B_8$}
\end{subfigure}
\begin{subfigure}[b]{0.19\textwidth}
    \centering
    \begin{tikzpicture}
        \node[styNode] (A) at (  0,  0) {};
        \node[styNode] (B) at ($(A) + (-90:\parh)$) {};
        \node[styNode] (C) at ($(B) + (-90:\parh)$) {};
        \node[styNode] (D) at ($(C) + (-90:\parh)$) {};

        \doubleEdge{A}{B}
        \doubleEdge{B}{C}
        \bendEdge{-30}{A}{D}
        \singleEdge{C}{D}
    \end{tikzpicture}
    \caption*{$B_9$}
\end{subfigure}
\vspace*{0.1cm}
\caption{The $13$ DAGs from the set $\mathcal{D}_4$.}
\label{fig:compGraphOfRVk3}
\end{figure}
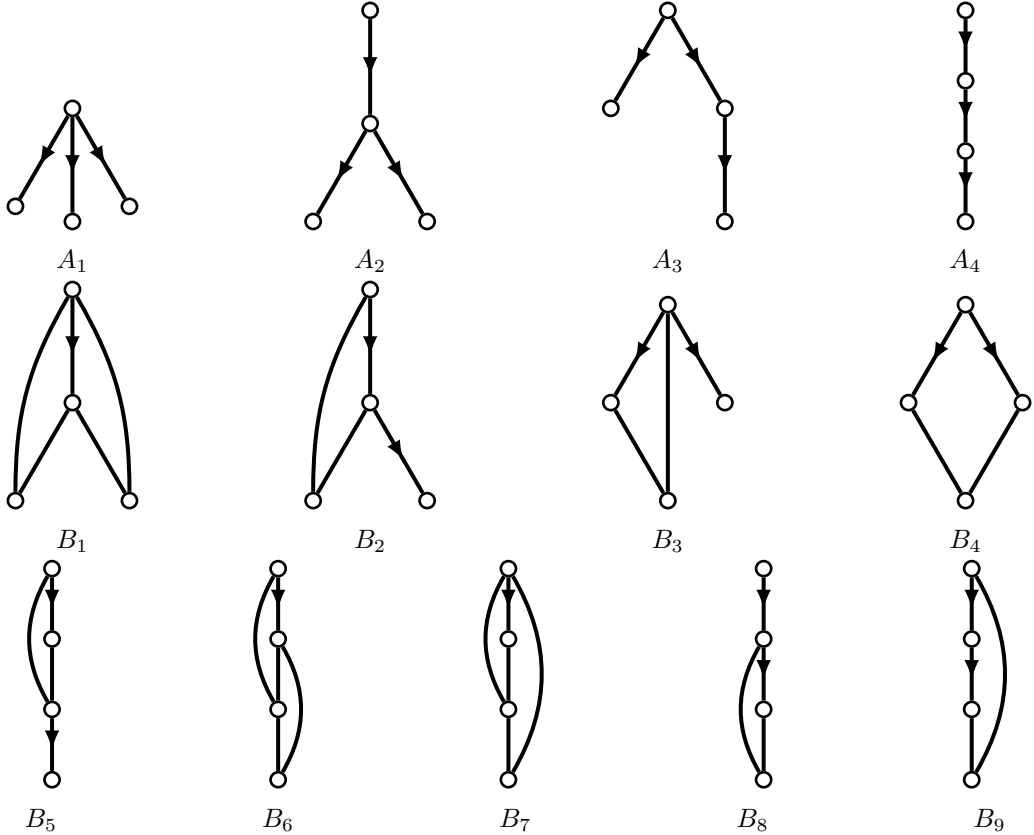

We next consider $k=3$. Here the set $\mathcal{D}_4$ has $13$ elements; see Figure~\ref{fig:compGraphOfRVk3}. Note that these DAGs fall into two types: the $A_j$'s are tree structures and the $B_j$'s are not. (The latter generate the reticulation-visible  networks which are not galled networks.) Next, we note that $m(A_1)=6, m(A_2)=m(B_1)=m(B_4)=2$ and the values are $1$ in all other cases. Thus, by Proposition~\ref{exp-RV}, we obtain for the exponential generating function $f_A(z)$ arising by the $A_j$'s,
\begin{multline*}
    f_A(z) = \frac{4z^3(29+12z+29z^2-37z^3+36z^4-14z^5)}{(1-2z)^{11/2}\,(1+\sqrt{1-2z})^3} \\
    + \frac{6z^3(3-z+7z^2-4z^3)}{(1-2z)^5\,(1+\sqrt{1-2z})^2} + \frac{2z^4}{(1-2z)^{9/2}\,(1+\sqrt{1-2z})},
\end{multline*}
and for the exponential generating function $f_B(z)$ arising from the $B_j$'s,
\[
f_B(z) = \frac{(1-X)^2(258-105X-153X^2-16X^3+26X^4+7X^5+3X^6-2X^7-2X^8)}{8\,X^{10}}.
\]
Then, by extracting coefficients of $f_A(z)+f_B(z)$ with Lemma~\ref{coef-Xd},
\begin{align*}
{\rm RV}_{\ell,3}=&\frac{4\ell^6+20\ell^5+33\ell^4-32\ell^3-76\ell^2+12\ell+12}{3}(2\ell-3)!!\\
&\qquad-2^{\ell-4}\frac{48\ell^4+175\ell^3+99\ell^2-262\ell-264}{3}\ell!.
\end{align*}

\section{Conclusion}\label{con}

In this paper, we derived exact and asymptotic counting results for the number of galled networks and reticulation-visible networks with a fixed number of reticulation vertices. For galled networks, we proposed a generating function approach which is based on the component graph method by Zhang and his coauthors. The exact results followed from it by coefficient extraction. Moreover, the asymptotic result was derived from it by methods from analytic combinatorics. For reticulation-visible networks, we also used the component graph method and generating functions to derive their exact counts. Moreover, the asymptotic result followed from previous work.

Thus, an asymptotic counting result for fixed $k$ is now known for all network classes from Figure~\ref{network-classes}; in fact, the asymptotic main term is the same for all these classes. This is slightly surprising for galled networks, since they are on a different ``branch" in the diagram in Figure~\ref{network-classes}. This suggests to consider phylogenetic networks which are both galled and satisfy the tree-child condition. This class of {\it galled tree-child networks} is also interesting from a theoretical point of view since very different methods have been used to determine the (asymptotic) numbers of tree-child and galled networks; see \cite{FuYuZh1,FuYuZh2}. We will do this in the forthcoming work \cite{ChFuYu}.

All the results in this paper were proved for fixed $k$. How about the number of networks if one sums over all possible values of $k$? The first order asymptotics of this number for galled networks is known and was obtained in \cite{FuYuZh2}. In fact, the latter also uses the component graph method. More precisely, the asymptotics was deduced from (\ref{form-GNl}). The analysis proceeded in two steps: first, the number of one-component galled networks was asymptotically studied by using the recurrence (\ref{rec-Mlk}). It turned out that $M_{\ell,k}$ has its maximum at $k=\ell$ (which is also the largest possible value of $k$) and satisfies a Poisson law. Then, in the second step, (\ref{form-GNl}) was used to find the asymptotics of ${\rm GN}_{\ell}$. This was possible because the first sum is over the set of (not necessarily binary) phylogenetic trees and the authors in \cite{FuYuZh2} found the trees whose contributions dominate the asymptotics.

In view of (\ref{form-RVl}) does this approach also work for reticulation-visible networks? First, we can skip the first step because, as mentioned in Section~\ref{cgm}, the class of one-component galled networks and one-component reticulation-visible networks coincide. Thus, in order to use the approach, we need to understand the class of (not necessarily binary) tree-child networks, which appears in the first sum of (\ref{form-RVl}). In particular, we need to know which of these networks contribute to the main term of the asymptotics.

For this, it would be good to have a better understanding of structural properties of this class. Binary tree-child networks have been counted in \cite{FuYuZh1}. However, generalizing these results to the non-binary case is non-trivial. For example, in \cite{ChFuLiWaYu1,ChFuLiWaYu2}, we considered $d$-combining tree-child networks which are tree-child networks with all internal vertices either bifurcating tree vertices or reticulation vertices with one child and exactly $d$ parents. Extending these results to general tree-child networks seems to be a major challenge. Such a generalization (and the structural knowledge it would entail) could, however, be helpful, if one wants to use (\ref{form-RVl}) to find the asymptotics of ${\rm RV}_{\ell}$.

\vspace*{-0.14cm}\section*{Acknowledgment}

We thank the two reviewers and the associate editor for helpful suggestions. Both authors acknowledge partially support by the National Science and Technology Council (NSTC), Taiwan under the grant NSTC-111-2115-M-004-002-MY2.

\vspace*{0.1cm}\section*{Appendix: Proof of Theorem~\ref{max-ret-RV}}

We use the component graph method from Section~\ref{cgm}. Since the component graph of a reticulation-visible network is a tree-child network with all vertices of indegree at most $2$ and no reticulation vertex  has just one child that is, moreover, a tree vertex, we first fix such a tree-child network $\tilde{C}$ (without arrows on the edges).

The maximal number of reticulation vertices a network decompressed from $\tilde{C}$ can have is obtained by placing arrows on all pendant edges except the ones directly below reticulation vertices. Let $r(\tilde{C})$ be the number of these edges plus the number of reticulation vertices of $\tilde{C}$ plus the number of internal vertices with exactly one incoming edge. Then, our goal is to find those ${\tilde{C}}$ which maximize $r(\tilde{C})$. Note that $r(\tilde{C})$ remains invariant if we replace vertices with indegree $2$ and outdegree at least $2$ by a reticulation vertex followed by a tree vertex and do not count the additional created edge. This is the set of $\tilde{C}$, we will consider in the sequel. (Thus, for the decompressing procedure, we first have to merge reticulation vertices followed by just one tree vertices, if there are any such vertices.)

We start with the following claim.

\vspace*{0.35cm}\noindent{\it Claim 1:} $r(\tilde{C})$ is maximized only for binary tree-child networks $\tilde{C}$.

\vspace*{0.3cm} Assume that $\tilde{C}$ has at least one vertex, say $v$, with outdegree $\geq3$. Replacing $v$ with a vertex of outdegree $2$ and attaching to it one of the children of $v$ and a vertex whose children are the remaining children of $v$ clearly gives a new tree-child network with $r(\tilde{C})$ increased by $1$ (since we have created a new internal vertex with indegree $1$). By iterating this procedure, we end up with a tree-child network which is binary. This proves our claim.

\newpage

%% place

%% \input{tikz/example_Maximal_Reticulated_RV_componentGraph_V2.txt}

\begin{figure}[t!]
\centering
\tikzset{
    styNode/.style={
        circle, draw, line width=1pt, inner sep = 2pt, font=\scriptsize
    },
    smallNode/.style={
        circle, draw, inner sep=0pt, minimum size=3pt%%, fill=black
    }
}
\newcommand{\doubleEdge}[2]{
    %%\draw[-{Latex[length=3mm]}, line width=3pt] (#1) -- (#2);
    %%\draw[white, line width=1pt, shorten >= 3mm] (#1) -- (#2);
    \draw[line width=1.2pt, decoration={markings,mark=at position 0.7 with {\arrow{latex[length=3mm]}}}, postaction={decorate}] (#1) -- (#2);
}
\newcommand{\singleEdge}[2]{
    %% \draw[-{Latex[length=2mm]}, line width=1.5pt] (#1) -- (#2);
    \draw[line width=1.2pt] (#1) -- (#2);
}
\newcommand{\bendEdge}[3]{
    %% \draw[-{Latex[length=2mm]}, line width=1.5pt] (#2) to [bend right=#1] (#3);
    \draw[line width=1.2pt] (#2) to [bend right=#1] (#3);
}
\newcommand{\parH}{1.5}
\newcommand{\parh}{0.5*\parH}

\newcommand{\myTheta}{30}
\newcommand{\addLeaf}[5]{
    %% add a leaf, (#2), labelled by #3 from the node (#1) at angle #4
    %% #1 .. label of A node
    %%\node[smallNode] (#2) at ($(#1) + (-90+#4:0.5*\parh)$) {};
    \node[line width=1pt,inner sep = 1pt, font=\scriptsize] (#2) at ($(#1) + (-90+#4:#5)$) {#3};
    \singleEdge{#1}{#2}
}
\newcommand{\addRetiLeaf}[5]{
    %% add a leaf, (#2), labelled by #3 from the node (#1) at angle #4
    %% #1 .. label of A node
    %%\node[smallNode] (#2) at ($(#1) + (-90+#4:0.5*\parh)$) {};
    \node[line width=1pt,inner sep = 1pt, font=\scriptsize] (#2) at ($(#1) + (-90+#4:#5)$) {#3};
    \doubleEdge{#1}{#2}
}

\begin{subfigure}[b]{0.44\textwidth}
    \centering
    \begin{tikzpicture}
        %% nodes
        %% A
        \node[smallNode] (a1) at (0,0) {};
        \node[smallNode] (a2) at ($(a1) + (-90:0.5*\parh)$) {};
        \node[smallNode] (a3) at ($(a2) + (-90+\myTheta:1.0*\parh)$) {};
        %% B
        \node[smallNode] (b1) at ($(a2) + (-90-\myTheta:1.5*\parh)$) {};
        \node[smallNode] (b2) at ($(b1) + (-90:0.5*\parh)$) {};
        \node[smallNode] (b3) at ($(b2) + (-90+\myTheta:1.0*\parh)$) {};
        %% C
        \node[smallNode] (c1) at ($(b2) + (-90-\myTheta:1.5*\parh)$) {};
        %% D
        \node[smallNode] (d1) at ($(a3) + (-90+\myTheta:2.5*\parh)$) {};

        %% single edges
        \foreach \x/\y in {
            a1/a2,a2/a3,
            a3/b1,b1/b2,b2/b3,
            b3/c1,
            a3/d1,b3/d1}{
            \singleEdge{\x}{\y}
        }
        %% bend edges
        \foreach \t/\x/\y in {
            1/a2/b1,1/b2/c1}{
            \bendEdge{\t*\myTheta}{\x}{\y}
        }
        %% leaves
        \foreach \x/\n/\t in {
            c1/2/0,
            d1/1/0}{
            \addLeaf{\x}{L\n}{\n}{\t*\myTheta}{0.7*\parh}
        }
        %% background
        \begin{scope}[on background layer={color=yellow}]
            \foreach \x in {
                a1, a2, a3,
                b1, b2, b3,
                d1}{
                \node[circle, fill=white!80!red, minimum size=12pt] at (\x) {};
            }
            \foreach \x/\y in {
                a1/a2,a2/a3,
                b1/b2,b2/b3}{
                \draw[draw=white!80!red, line width=12pt] (\x.center) -- (\y.center);
            }
        \end{scope}
    \end{tikzpicture}
    \caption*{$N$}
\end{subfigure}
\renewcommand{\myTheta}{30}
\newcommand{\deltaH}{0.2}
\begin{subfigure}[b]{0.44\textwidth}
    \centering
    \begin{tikzpicture}
        %% nodes
        \node[styNode] (A) at (  0,  0) {};
        \node[styNode] (B) at ($(A) + (-90-\myTheta:\parH)$) {};
        %\node[styNode] (C) at ($(B) + (-90-\myTheta:\parH)$) {C};
        \node[styNode] (D) at ($(A) + (-90+\myTheta:1.5*\parH)$) {};
        %% double edges
        \foreach \x/\y in {
            A/B}{
            \doubleEdge{\x}{\y}
        }
        %% single edges
        \foreach \x/\y in {
            A/D,B/D}{ %%
            \singleEdge{\x}{\y}
        }
        %% leaves
        \foreach \x/\n/\t in {
            D/1/0}{
            \addLeaf{\x}{L\n}{\n}{\t*\myTheta}{\parh}
        }
        %% Rericulations leaves
        \foreach \x/\n/\t in {
            B/2/0}{
            \addRetiLeaf{\x}{L\n}{\n}{\t*\myTheta}{\parh}
        }
        %% bend edges
        %%\foreach \t/\x/\y in {
        %%    1/A/H,1/B/E,-1/C/I}{
        %%    \bendEdge{\t*\myTheta}{\x}{\y}
        %%}

        %% to controll vertically aligned
        \node[] (TOP) at ($(A) + ( 90:\deltaH)$) {};
        \node[] (BTM) at ($(L1) + (-90:\deltaH)$) {};

    \end{tikzpicture}
    \caption*{$C(N)$}
\end{subfigure}
\vspace*{0.15cm}\caption{Left: The smallest example of a maximal reticulated reticulation-visible network with $2$ leaves (with the tree-components highlighted). Right: The component graph of the network; note that it is a maximal reticulated binary tree-child network.}
\label{fig:example_maximal_reticulated_RV_V2}
\end{figure}
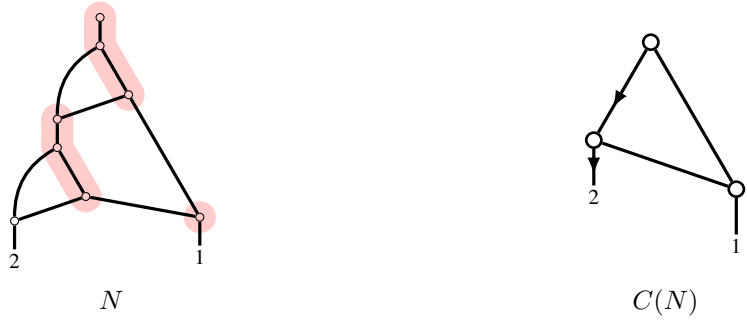

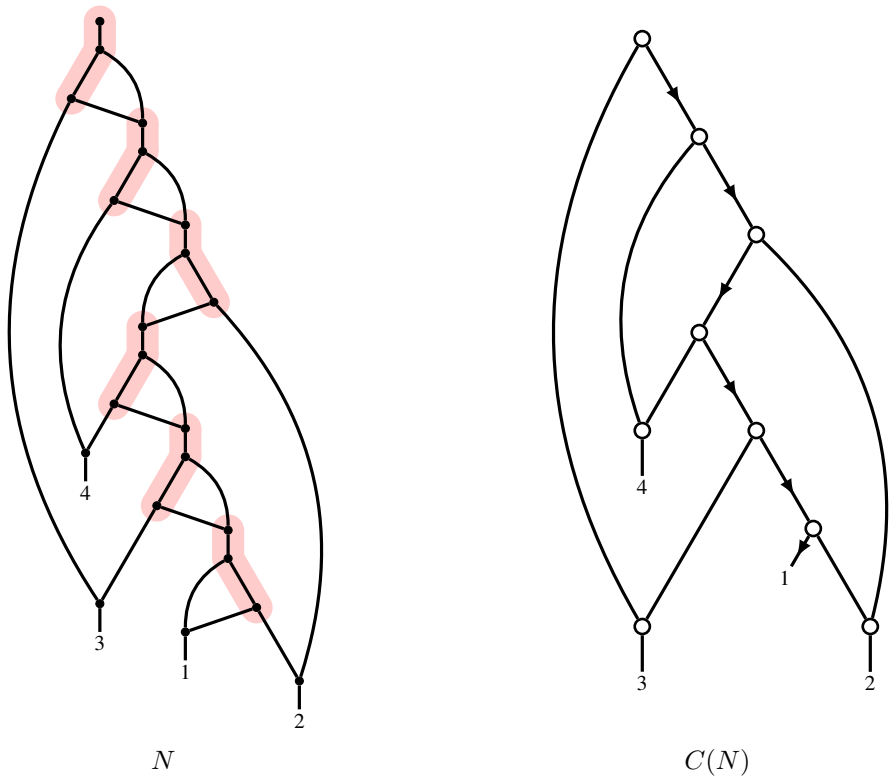
\begin{figure}[H]
\centering
\tikzset{
    styNode/.style={
        circle, draw, line width=1pt, inner sep = 2pt, font=\scriptsize
    },
    smallNode/.style={
        circle, draw, inner sep=0pt, minimum size=3pt, fill=black
    }
}
\newcommand{\doubleEdge}[2]{
    %%\draw[-{Latex[length=3mm]}, line width=3pt] (#1) -- (#2);
    %%\draw[white, line width=1pt, shorten >= 3mm] (#1) -- (#2);
    \draw[line width=1.2pt, decoration={markings,mark=at position 0.7 with {\arrow{latex[length=3mm]}}}, postaction={decorate}] (#1) -- (#2);
}
\newcommand{\singleEdge}[2]{
    %% \draw[-{Latex[length=2mm]}, line width=1.5pt] (#1) -- (#2);
    \draw[line width=1.2pt] (#1) -- (#2);
}
\newcommand{\bendEdge}[3]{
    %% \draw[-{Latex[length=2mm]}, line width=1.5pt] (#2) to [bend right=#1] (#3);
    \draw[line width=1.2pt] (#2) to [bend right=#1] (#3);
}
\newcommand{\parH}{1.5}
\newcommand{\parh}{0.5*\parH}

\newcommand{\myTheta}{30}
\newcommand{\addLeaf}[5]{
    %% add a leaf, (#2), labelled by #3 from the node (#1) at angle #4
    %% #1 .. label of A node
    %%\node[smallNode] (#2) at ($(#1) + (-90+#4:0.5*\parh)$) {};
    \node[line width=1pt,inner sep = 1pt, font=\scriptsize] (#2) at ($(#1) + (-90+#4:#5)$) {#3};
    \singleEdge{#1}{#2}
}
\newcommand{\addRetiLeaf}[5]{
    %% add a leaf, (#2), labelled by #3 from the node (#1) at angle #4
    %% #1 .. label of A node
    %%\node[smallNode] (#2) at ($(#1) + (-90+#4:0.5*\parh)$) {};
    \node[line width=1pt,inner sep = 1pt, font=\scriptsize] (#2) at ($(#1) + (-90+#4:#5)$) {#3};
    \doubleEdge{#1}{#2}
}

\begin{subfigure}[b]{0.44\textwidth}
    \centering
    \begin{tikzpicture}
        %% nodes
        %% A
        \node[smallNode] (a1) at (0,0) {};
        \node[smallNode] (a2) at ($(a1) + (-90:0.5*\parh)$) {};
        \node[smallNode] (a3) at ($(a2) + (-90-\myTheta:1.0*\parh)$) {};
        %% B
        \node[smallNode] (b1) at ($(a2) + (-90+\myTheta:1.5*\parh)$) {};
        \node[smallNode] (b2) at ($(b1) + (-90:0.5*\parh)$) {};
        \node[smallNode] (b3) at ($(b2) + (-90-\myTheta:1.0*\parh)$) {};
        %% C
        \node[smallNode] (c1) at ($(b2) + (-90+\myTheta:1.5*\parh)$) {};
        \node[smallNode] (c2) at ($(c1) + (-90:0.5*\parh)$) {};
        \node[smallNode] (c3) at ($(c2) + (-90+\myTheta:1.0*\parh)$) {};
        %% D
        \node[smallNode] (d1) at ($(c2) + (-90-\myTheta:1.5*\parh)$) {};
        \node[smallNode] (d2) at ($(d1) + (-90:0.5*\parh)$) {};
        \node[smallNode] (d3) at ($(d2) + (-90-\myTheta:1.0*\parh)$) {};
        %% E
        \node[smallNode] (e1) at ($(d3) + (-90-\myTheta:1.0*\parh)$) {};
        %% F
        \node[smallNode] (f1) at ($(d2) + (-90+\myTheta:1.5*\parh)$) {};
        \node[smallNode] (f2) at ($(f1) + (-90:0.5*\parh)$) {};
        \node[smallNode] (f3) at ($(f2) + (-90-\myTheta:1.0*\parh)$) {};
        %% G
        \node[smallNode] (g1) at ($(f2) + (-90+\myTheta:1.5*\parh)$) {};
        \node[smallNode] (g2) at ($(g1) + (-90:0.5*\parh)$) {};
        \node[smallNode] (g3) at ($(g2) + (-90+\myTheta:1.0*\parh)$) {};
        %% H
        \node[smallNode] (h1) at ($(f3) + (-90-\myTheta:2.0*\parh)$) {};
        %% I
        \node[smallNode] (i1) at ($(g2) + (-90-\myTheta:1.5*\parh)$) {};
        %% J
        \node[smallNode] (j1) at ($(g3) + (-90+\myTheta:1.5*\parh)$) {};

        %% single edges
        \foreach \x/\y in {
            a1/a2,a2/a3,
            a3/b1,b1/b2,b2/b3,
            b3/c1,c1/c2,c2/c3,
            c3/d1,d1/d2,d2/d3,
            d3/e1,
            d3/f1,f1/f2,f2/f3,
            f3/g1,g1/g2,g2/g3,
            f3/h1,
            g3/i1,
            g3/j1}{
            \singleEdge{\x}{\y}
        }

        %% bend edges
        \foreach \t/\x/\y in {
            -1/a2/b1,
            -1/b2/c1,
            1/c2/d1,
            -1/d2/f1,
            -1/f2/g1,
            1/g2/i1,
            1/a3/h1,
            1/b3/e1,
            -1/c3/j1}{
            \bendEdge{\t*\myTheta}{\x}{\y}
        }

        %% leaves
        \foreach \x/\n/\t in {
            e1/4/0,
            h1/3/0,
            i1/1/0,
            j1/2/0}{
            \addLeaf{\x}{L\n}{\n}{\t*\myTheta}{0.7*\parh}
        }

        %% background
        \begin{scope}[on background layer={color=yellow}]
            \foreach \x in {
                a1, a2, a3,
                b1, b2, b3,
                c1, c2, c3,
                d1, d2, d3,
                f1, f2, f3,
                g1, g2, g3}{
                \node[circle, fill=white!80!red, minimum size=12pt] at (\x) {};
            }
            \foreach \x/\y in {
                a1/a2,a2/a3,
                b1/b2,b2/b3,
                c1/c2,c2/c3,
                d1/d2,d2/d3,
                f1/f2,f2/f3,
                g1/g2,g2/g3}{
                \draw[draw=white!80!red, line width=12pt] (\x.center) -- (\y.center);
            }
        \end{scope}
    \end{tikzpicture}
    \caption*{$N$}
\end{subfigure}
\renewcommand{\myTheta}{30}
\newcommand{\deltaH}{0.5}
\begin{subfigure}[b]{0.44\textwidth}
    \centering
    \begin{tikzpicture}
        %% nodes
        \node[styNode] (A) at (  0,  0) {};
        \node[styNode] (B) at ($(A) + (-90+\myTheta:\parH)$) {};
        \node[styNode] (C) at ($(B) + (-90+\myTheta:\parH)$) {};
        \node[styNode] (D) at ($(C) + (-90-\myTheta:\parH)$) {};
        \node[styNode] (E) at ($(D) + (-90-\myTheta:\parH)$) {};
        \node[styNode] (F) at ($(D) + (-90+\myTheta:\parH)$) {};
        \node[styNode] (G) at ($(F) + (-90+\myTheta:\parH)$) {};
        \node[styNode] (H) at ($(F) + (-90-\myTheta:2*\parH)$) {};
        \node[styNode] (I) at ($(G) + (-90+\myTheta:\parH)$) {};

        %% double edges
        \foreach \x/\y in {
            A/B,B/C,C/D,D/F,F/G}{
            \doubleEdge{\x}{\y}
        }

        %% single edges
        \foreach \x/\y in {
            D/E,
            F/H,
            G/I}{ %%
            \singleEdge{\x}{\y}
        }

        %% leaves
        \foreach \x/\n/\t in {
            E/4/0,
            H/3/0,
            I/2/0}{
            \addLeaf{\x}{L\n}{\n}{\t*\myTheta}{\parh}
        }
        %% Rericulations leaves
        \foreach \x/\n/\t in {
            G/1/-1}{
            \addRetiLeaf{\x}{L\n}{\n}{\t*\myTheta}{\parh}
        }

        %% bend edges
        \foreach \t/\x/\y in {
            1/A/H,1/B/E,-1/C/I}{
            \bendEdge{\t*\myTheta}{\x}{\y}
        }

        %% to controll vertically aligned
        \node[] (TOP) at ($(A) + ( 90:\deltaH)$) {};
        \node[] (BTM) at ($(L3) + (-90:\deltaH)$) {};

    \end{tikzpicture}
    \caption*{$C(N)$}
\end{subfigure}
\vspace*{0.15cm}\caption{A maximal reticulated reticulation-visible network with $4$ leaves and $9$ reticulation vertices. The component graph is a maximal reticulated binary tree-child network. Note that $N$ is the only network obtained by decompressing $C(N)$.}
\label{fig:example_maximal_reticulated_RV_V1}
\end{figure}

\vspace*{0.35cm}\noindent{\it Claim 2:} The maximum of $r(\tilde{C})$ over the set of all binary tree-child networks is $3\ell-3$ and this bound is achieved if and only if $\tilde{C}$ is a maximal reticulated binary tree-child network.

\vspace*{0.3cm} Note that in any binary tree-child network, we have $\ell+k=t+2$, where $t$ is the number of tree vertices; see Section~1 in \cite{FuYuZh1}. Thus,
\[
r(\tilde{C})=\ell+k+t-k=2\ell+k-2.
\]
Since $k\leq\ell-1$ (see Section~\ref{intro}) with this bound achieved exactly by the maximal reticulated tree-child networks, the claim follows.

Overall, we have proved so far that the maximal number of reticulation vertices of a reticulation-visible network is $3\ell-3$ and this bound is achieved if and only if the component graph of the network is a maximal reticulated binary tree-child network. Finally, the tree vertices of these networks have one child which is a reticulation vertex and one child which is not; see Lemma~1 in \cite{FuYuZh1}. Thus, they are replaced by a one-component network which has $2$ leaves exactly one of which is below a reticulation vertex. However, the number of these one-component networks is $M_{2,1}=1$. Consequently, the decompression of every maximal reticulated binary tree-child network gives exactly one reticulation-visible network; see Figure~\ref{fig:example_maximal_reticulated_RV_V2} for the smallest example (see also \cite{BoSe}) and Figure~\ref{fig:example_maximal_reticulated_RV_V1} for a larger example.\qed
\end{document}